\documentclass[11pt]{amsart}

\usepackage{mathptmx}
\usepackage{amsmath,amssymb,amsthm}
\usepackage{graphics}
\usepackage{enumerate}
\usepackage[dvips]{color}
\usepackage{version}
\usepackage{stmaryrd}
\usepackage[all]{xy}

\usepackage{todonotes}

\newtheorem{Thm}{Theorem}[section] 
\newtheorem{Lem}[Thm]{Lemma}

\newtheorem{Cor}[Thm]{Corollary}
\newtheorem{Cor.Conj}[Thm]{Corollary of Conjecture}

\newtheorem{Prop}[Thm]{Proposition}

\newtheorem{Conj}[Thm]{Conjecture}

\newtheorem{Ques}[Thm]{Question}

\newtheorem{Ass}[Thm]{Assumption}

\theoremstyle{remark}
\newtheorem{Rem}[Thm]{Remark}
\newtheorem{Ex}[Thm]{Example}

\theoremstyle{definition}
\newtheorem{Def}[Thm]{Definition}

\newtheorem{Step}{Step}
\newtheorem{Stp}{Step}

\newtheorem{Const}[Thm]{Construction}

\newcommand{\Spec}{\mathop{\mathrm{Spec}}\nolimits}

\newcommand{\Frac}{\mathop{\mathrm{Frac}}\nolimits}
\newcommand{\Supp}{\mathop{\mathrm{Supp}}\nolimits}

\newcommand{\PTrop}{\mathop{\mathrm{PTrop}}\nolimits}

\newcommand{\Val}{\mathop{\mathrm{Val}}\nolimits}

\newcommand{\R}{\ensuremath{\mathbb{R}}}
\newcommand{\C}{\ensuremath{\mathbb{C}}}

\newcommand{\Z}{\ensuremath{\mathbb{Z}}}

\newcommand{\Q}{\mathbb{Q}}

\newcommand{\A}{\mathbb{A}}

\newcommand{\X}{\mathcal{X}}

\newcommand{\DD}{\mathcal{D}}
\newcommand{\M}{\mathcal{M}}

\begin{document}

\title[On degenerations and moduli of Calabi-Yau varieties]{
Degenerated  
Calabi-Yau varieties with infinite components, moduli compactifications, \\ 
and limit toroidal structures }
\author{Yuji Odaka}
\date{\today}

\maketitle

\begin{abstract}
For any degenerating Calabi-Yau family, 
we introduce new limit space which we call {\it galaxy}, 
whose dense subspace is 
the disjoint union of {\it countably infinite} open Calabi-Yau varieties, 
parametrized by the rational points of the Kontsevich-Soibelman's 
essential skeleton, 
while dominated by the Huber adification over the Puiseux series field. 

Other topics include: projective limits of toroidal compactifications 
(\S \ref{lmmod.sec}), locally modelled on 
{\it limit toric varieties} 
(\S\ref{lim.toric.sec}), the way to attach tropicalized family to given Calabi-Yau 
family (\S\ref{fiber.mod.sec}), 
which are weakly related to each other. 
\end{abstract}


\section{Introduction}

Degenerations and moduli spaces 
of polarized Calabi-Yau varieties\footnote{We use this term 
as equivalent of K-trivial varieties in the broad sense, i.e., 
normal 
projective varieties with numerically trivial canonical divisors.} 
have been attracting researchers over decades, especially since 
they lie among many areas of research, 
such as algebraic geometry, differential geometry, 
mirror symmetry, arithmetic geometry, 
non-archimedean geometry, tropical geometry, 
and others.  

\subsection{Partial summary}
This paper introduces and discusses various degeneration structure and compactifications, 
partially motivated by their own beauty and relations with non-archimedean geometry (and tropical geometry). 
Topics in each section (including appendices) 
are only weakly connected, so the interested readers could directly skip to their own concerned part. 

\vspace{2mm}

\subsubsection*{Contents of \S \ref{fiber.sec} (Galaxy models)}
Our first construction in \S \ref{fiber.sec} is as follows: 
for any degenerating one parameter family of (klt) Calabi-Yau varieties, 
we consider a new kind of degeneration as an intricate connected 
locally ringed space which we call {\it galaxy}. 
Here is the simplest example: 

\begin{Ex}[Elliptic curve case]\label{ell.ex} 
For a minimal degeneration of elliptic curves $\X\to \Delta\ni 0$ 
of $I_{m}$-type ($m\in \mathbb{Z}_{>0}$) in the sense of Kodaira i.e., 
with the central fiber $m$-gon consisting of transversally 
intersecting  $\mathbb{P}^{1}$s, take a finite base change 
$\Delta'\to \Delta$ ramifying at $0$ of degree $d$ and 
consider the fiber product $\X\times_{\Delta} \Delta'$. 
Then, after the minimal resolutions of $m$ $A_{d-1}$-singularities 
in $\X\times_{\Delta} \Delta'$, we obtain 
$I_{dm}$-type degeneration of (essentially same) elliptic curves. 
If we take a sequence of the degree 
$d$ diverging in the divisible order, we obtain a 
projective system of $I_{dm}$ fibers i.e., 
$md$-gon of $\mathbb{P}^{1}$s where $m$ varies. 
Its projective limit $X_{\infty}$ with respect to the degree $d$s, which is the easiest example of 
galaxy, includes 
infinite $(\mathbb{P}^{1}\setminus \{0,\infty\})$s parametrized by 
$\Q/\Z= \cup_{d} \frac{1}{dm}\Z /\Z$ and closed points bijectively 
corresponding to the fractional part of 
irrational numbers i.e., 
parametrized by $(\R\setminus \Q)/\Z$. 
\end{Ex}

In \S \ref{fiber.sec}, 
we extend the above space ``galaxy'' to arbitraliry dimensional 
Calabi-Yau varieties and reveal the basic structures roughly 
as follows. 
Here, $\Delta\ni 0$ means the germ of a 
smooth pointed $k$-curve, while we base change over 
the Puiseux series field 
$k((t^{\Q}))$ which is the fraction field of 
$k[[t^{\Q}]]:=\cup_{d\ge 1}k[[t^{1/d}]]$. 
See the details for later sections. 

\begin{Thm}[cf., \ref{subdiv2}, \ref{limit.defs}, \S \ref{rel.skel}]\label{galaxy.intro}
For any punctured meromorphic 
family of klt projective Calabi-Yau varieties 
$\X^{*}\to \Delta^{*}=\Delta\setminus \{0\}$, 
its base change to ${\rm Spec}k((t^{\Q}))$ has a 
model called 
{\it galaxy model} $\X_{\infty}$, with the central fiber 
$X_{\infty}$ (we call it galaxy), 
such that the following holds: 

\begin{enumerate}
\item \label{map.to.KS}
There is a canonical continuous surjective map 
\begin{align}
f_{tr}\colon 
X_{\infty} \twoheadrightarrow \Delta^{KS}(\X_{\eta})=B,
\end{align}
where $\Delta^{KS}(\X_{\eta})$ is 
{\bf the essential skeleton} (\cite{KS}) 
inside the Berkovich analytification $\X_{\eta}^{*, an}$ of the 
generic fiber $\X_{\eta}^{*}$. 

\item 
For any point $x\in \Delta(\X_{0})$, 
whose coordinates with respect to the natural integral affine 
structure are {\bf rational}, $f_{tr}^{-1}(x)$ includes as an open 
dense subset 
the  
complex analytification of {\bf open Calabi-Yau varieties} (each of which is unique up to log crepant birational maps). 

\item 
There is a continous surjective map 
\begin{align}
\X_{\eta, k((t^{\Q}))}^{*,ad}\twoheadrightarrow 
X_{\infty}
\end{align} from 
{\bf the Huber adification} $\X_{\eta, k((t^{\Q}))}^{*,ad}$ 
(naturally associated adic space in the sense of \cite{Hub96}) of the base change $\X_{\eta, k((t^{\Q}))}^{*}$ 
of the generic fiber $\X^{*}_{\eta}$ to the 
Puiseux series field $k((t^{\Q}))$. 

\end{enumerate}
\end{Thm}
These galaxy models 
be seen as analogue of ``minimal'' models over non-finite type 
spectra $k[[t^{\Q}]]$. 
Recall that in the elliptic curve case (Example \ref{ell.ex}), 
the essential skeleta is $\R/\Z \simeq S^{1}$ (the {\it tropical elliptic curve}), and the open Calabi-Yau variety appearing in 
the above theorem is $\mathbb{P}^{1}\setminus\{0,\infty\}$. 
Note that any $I_{md}$-degenerate fiber does {\it not} admit 
a continous map onto $\R/\Z$ nor even to the subset 
$\frac{1}{md}\Z/\Z$, 
hence the above \eqref{map.to.KS} 
is an effect of taking the projective limit. 

In particular, we observe the following basic structure of the 
galaxies $X_{\infty}=\varprojlim_{i} X_{i}$. 

\begin{Cor}[=Corollary \ref{cor.subdiv.lc}: Decomposition of Galaxies]\label{cor.subdiv.lc.intro}
Any galaxy $X_{\infty}$ 
it has a natural decomposition into open part and the closed part: 
\begin{align}\label{decomp.lim}
X_{\infty}=(\sqcup_{a\in B(\Q)} \{\text{open klt log Calabi-Yau variety } U(a)\})
\bigsqcup X_{\infty}^{NKLT}.
\end{align}
Here, $B$ denotes the essential skeleton (as \ref{galaxy.intro} again)  
 and $B(\Q)$ means 
its rational points with respect to 
the $\Q$-affine structure, which does not depend on $i$ by 
Theorem \ref{dlt.lem} \eqref{dual.cpx.same} and 
$X_{\infty}^{NKLT}:=\varprojlim_{i}X_{i}^{NKLT}$, 
where NKLT stands for the non-klt closed loci. 
\end{Cor}

\vspace{6mm}

\subsubsection*{Contents of \S \ref{lmmod.sec} (Limit toroidal 
compactifications)}

In \S \ref{lmmod.sec} over $k=\C$ in turn, we discuss on the projective limit of toroidal compactifications of \cite{AMRT} and 
its analogue for more general (moduli) varieties. 
It sounds somewhat independent from \S \ref{fiber.sec} 
but here we observe an 
analogous phenomenon of the above theorem \ref{galaxy.intro}, 
especially the natural continuous maps to their tropical versions.

More precisely, what we do there is as follows. 
Fix a locally Hermitian symmetric space $M$ of non-compact type, 
i.e., of the ubiquitous form $M=\Gamma\backslash G/K$ where $G$ 
is a real valued points of a simple algebraic group over $\Q$, 
$K$ its maximal compact subgroup with one dimensional center, 
$\Gamma$ an arithmetic discrete subgroup of $G$. 
Some renowned examples are the moduli space of 
$g$-dimensional principally polarized abelian varieties or 
that of primitively 
polarized K3 surfaces with possibly ADE singularities 
(of fixed genera). 

Recall that for certain combinatorial data i.e., 
admissible collection of fans $\Sigma=\{\Sigma(F)\}$, 
there is an associate toroidal compactification 
$\overline{M}^{\rm tor,\Sigma}$ 
of $M$ 
constructed in \cite{AMRT}, 
and its complex analytification 
$\overline{M}^{\rm tor,\Sigma,an}$. 

Now, we consider and introduce 
the projective limit of all of its toroidal 
compactifications as a locally ringed space and call 
the {\it limit toroidal compactification}: 
\begin{align}
\overline{M^{\rm tor, \infty}}^{an}&:=\varprojlim_{\Sigma} \overline{M}^{\rm tor,\Sigma,an}.
\end{align}

More precisely, the ingredient $\overline{M}^{\rm tor,\Sigma,an}$ 
of the right hand side is the complex analytification of the 
toroidal compactification with respect to the combinatorial data i.e., 
the admissible collection of fans 
$\Sigma$ (\cite{AMRT}). 
On the other hand, we also recall 
$\overline{M}^{\rm MSBJ}$, the minimal 
Morgan-Shalen-Boucksom-Jonsson compactification (\cite[Appendix]{TGC.II},\cite[\S 2]{OO}) i.e., 
the MSBJ compactification corresponding to  
the toroidal compactifications \cite{AMRT} 
(which do not depend on the combinatorial data i.e., 
the admissible collection of rational 
polyhedra as \cite[Theorem 2.1]{OO} shows). 

\begin{Thm}[=Theorem \ref{dom.TGC}]
Then, 
there is a natural continuous surjective map 
\begin{align}
\phi_{tr}\colon \overline{M^{\rm tor,\infty}}^{an}\to \overline{M}^{\rm MSBJ}
\end{align}
which extends the identity map on $M$. 

For a general point $x$ of the boundary 
$\partial \overline{M}^{\rm MSBJ}$, 
the fiber $\phi_{tr}^{-1}(x)$ coincides with the 
{\bf limit toric variety} 
$\overline{T}^{n-r, an}_{\infty}$ 
we 
introduce at section \S \ref{lim.toric.sec}. 
\end{Thm}
Roughly speaking, the {\it limit toric variety} above 
is the projective limit of all proper toric varieties (of the 
fixed dimension) which we study in \S \ref{lim.toric.sec}. 

Also, in \S \ref{lmmod.sec} we discuss the Zariski-Riemann 
type compactification and clarify relations with the above 
compactifications. 

\vspace{2mm}

\subsubsection*{Contents of \S \ref{fiber.mod.sec} (Attaching tropical family to varieties family)}

Until the section \ref{lmmod.sec}, we have discussions of the 
degenerations and moduli compactifications independently although 
they are both connected with tropicalizing phenomena. 
\vspace{2mm}

The purpose of \S \ref{fiber.mod.sec} 
is to fill this gap to some extent by attaching a family of 
``tropical varieties'' to a family of varieties on the complex moduli. 
A little more precisely, 
we construct a family, with a connected total space, over the minimal Morgan-Shalen-Boucksom-Jonsson 
compactification of the moduli spaces of varieties which is a flat 
family of polarized varieties in the open locus, whereas at the boundary it is a family of 
``corresponding tropical varieties''. 
We give a rough vague statements as follows, which is essentially 
known by \cite[\S 8]{ACP}, 
\cite{CCUW}, \cite[\S 9 B]{Nam} etc. 
The precise meanings are still left until \S \ref{fiber.mod.sec}. 

\begin{Prop}[cf., \S \ref{fiber.mod.sec} for precise meanings]
Over the minimal MSBJ compactification $\overline{M}^{\rm MSBJ}$ 
of the moduli of 
$M=M_{g}$ (resp., $M=A_{g}$), we have a continous family 
which is the original family of hyperbolic curves on $M=M_{g}$ (resp., 
principally polarized $g$-dimensional abelian varieties on $M=A_{g}$) 
$M$ (at stacky level!) 
and the family of tropical curves (resp., tropical abelian varieties) 
on the boundary $\partial \overline{M}^{\rm MSBJ}$. 
\end{Prop}

\vspace{2mm}

We hope that our newly introduced structures and problems 
can be of its own interests in non-archimedean geometry and 
tropical geometry among others. 

\vspace{2mm}

\subsubsection*{On the appendices}
Also, as additional notes, 
Appendix \S \ref{appendixA} discusses Morgan-Shalen type compactifications for  
general Berkovich analytic spaces, their basic properties such as 
funtoriality with respect to morphisms. 
These are partially used in a few places in 
our main contents. 

Appendix \S \ref{SBB.sec} discusses possible analogues of 
classical Satake-Baily-Borel compactifications 
(section~\ref{SBB.sec}) in the more moduli theoritic contexts, in the spirit of Griffiths 
\cite{Griffiths} for period spaces, followed by \cite{KU, GGLR} etc, 
but with more focus on the relations with degenerations and moduli. 

\vspace{5mm}

\subsection{Some background and history}

Before going into the main contents, 
we also review some historical background 
especially from birational geometry and the asymptotic behaviour of 
Ricci-flat K\"ahler metrics, although unfortunately 
differential geometric 
perspectives have not yet been substantially developped. 

\subsubsection*{Differential geometric background}

The well-known work of Yau \cite{Yau} showed the existence of 
unique Ricci-flat K\"ahler metric on polarized Calabi-Yau varieties. 
Since then, a natural question has been to consider 
asymptotic behaviour of the metric with respect to the 
variation of the polarized Calabi-Yau varieties in concern. 

Sometimes the metrics collapse i.e., 
roughly the dimension of the limit is less, whereas sometimes not. 
Non-collapsing situation is well-understood by now, 
due to \cite{Wan97, Wan03, DS, Tos.WP, Tak}. 

In the collapsing cases, the study are started 
in \cite{KS, GW} etc and developped in many literatures such as 
\cite{GTZ1, GTZ2, BJ} among others. 
The metrics  behaviour in those two cases are observed to be 
very different. 

However, above works are for one parameter family, either 
along a holomorphic family or adiabatic limits. 
Such framework is recently extended to general sequences or equivalently 
to the whole moduli in the (ongoing) series of works 
\cite{Od.Ag, OO, PL, Osh}. 
Above works mainly concern collapsing of
the {\it rescaled} Ricci-flat K\"ahler metrics with fixed 
{\it diameters}. 

\vspace{2mm}

We whereas expect our consideration of analogue of 
Satake-Baily-Borel compactifications 
in \S \ref{SBB.sec} of this paper could be 
related to the Gromov-Hausdorff 
compactifications with respect to different rescale i.e., 
{\it with fixed volumes 
of the Ricci-flat metrics}, in a certain sense. 
See Question \ref{SBB.GH} in Appendix 
\S \ref{SBB.sec}. 

On the other hand, differential geometric meanings of 
other compactifications we discuss 
in \S \ref{lmmod.sec} is more unclear and nontrivial yet. 
We leave such directions of research for future. 

\subsubsection*{Algebro-geometric background}
Motivated by the projective moduli construction of ample 
(log-)canonical class varieties with semi-log-canonical 
singularities (``KSBA moduli'' cf., 
\cite{SB, KSB, Ale94, Ale1, Ale2, Kol}), 
as well as many other 
explicit moduli spaces 
(e.g., \cite{Nam1, Shah, AN, Ale02, Nak1, Nak2}) 
and a lot of great differential geometric 
works such as \cite{BKN, Fuj92, FS, Don04}, 
\cite[\S5, Conjecture~5.2]{Od10} 
discussed the possibility of constructing theory of 
at least partially compact 
moduli spaces via K-stability 
for more general polarised varieties 
as ``K-moduli'': moduli of K-(poly)stable varieties. 
Indeed, there are many ways from different fields to lead to this 
expectation, and the above expression of review should be still 
with personal biase of my own perspectives. Here, we 
make some precise but certainly weaker versions of the K-moduli conjecture 
in the case of Calabi-Yau varieties, and confirm some progresses 
which have been done. 

The currently most recognized approach for more than a decade, for 
compactifying the moduli of polarized Calabi-Yau varities has been 
essentially 
``log KSBA'' i.e., 
to attach {\it (extra) ample divisors} 
as boundary to the Calabi-Yau varieties in concern, to 
make natural {\it logarithmic} generalization (i.e., with boundary divisors) of the KSBA approach 
work. Lately, such approach gave various interesting compactifications and 
their explicit description e.g., \cite{Ale02,Laza2,AET,ABE} among others. 
However, in general, such approach certainly gives either {\it non-}uniqueness 
of the compactification due to the additional data of choosing the divisors and 
change the framework of discussion somewhat. 
In this paper, together with \cite{ops}, we still seek for {\it canonical} ``algebro-''
geometric compactifications without adding divisors or {\it canonical} 
``algebro-''geometric limits of degenerations. 

\vspace{3mm}
Before going into details, for comparison, let us review another story 
from the K-stability perspectives. 
Especially since 2012, 
the construction of K-moduli 
construction in the case of anticanonically polarized 
$\mathbb{Q}$-Fano varieties has developed: 
projective moduli spaces of 
K-polystable (K\"ahler-Einstein) $\mathbb{Q}$-Fano varieties 
constructed under $\Q$-Gorenstein smoothability condition 
cf., \cite{OSS, SSY, LWX, Od15}, which are 
{\it \'etale locally} 
GIT quotients (called ``good moduli spaces'' by Alper). 
More recently, 
there has been also nice developments to algebraize the arguments in 
the construction, so that it naturally extends to more general sigular 
$\mathbb{Q}$-Fano varieties, we do not try to make 
complete list of such references here, with apology. 
In any case, the notable feature in the anti-canonically polarized 
case is that for any punctured family of K\"ahler-Einstein K-polystable $\Q$-Fano varieties, we can conjecturally fill in a 
K-polystable $\Q$-Fano variety to ensure the compactness of the moduli. 

On the other hand, the case of our focus in this paper, 
polarized varieties with numerically trivial 
log-canonical class (``Calabi-Yau''), 
K-polystable ones do {\it not} form compact moduli unlike 
$\Q$-Fano case unfortunately: 
even for elliptic curve case, the nodal minimal degeneration, i.e., 
$I_{v}$-type in the sense of Kodaira is {\it not} K-polystable. 
Indeed, recall that 

\begin{Thm}[{\cite[6.3]{OS}}]\label{OSrev}
Assume $(X, D)$ is a log smooth 
Calabi-Yau pair, i.e., $K_X+D=0$ with a polarization $L$. 
Then,  $((X,D),L)$ 
is log K-semistable 
if and only if coefficients of $D$ are at most $1$ 
but it can{\it -not} be log K-polystable unless $\lfloor 
D\rfloor =0$. 
More generally, semi-log-canonical Calabi-Yau polarized pair 
$((X,D),L)$ is log K-semistable but it is log K-polystable 
if and only if the underlying log pair $(X,D)$ has only klt singularities. 
\end{Thm}

Furthermore, as is well-known, for a fixed 
punctured family of 
Calabi-Yau varieties e.g., K3 surfaces, 
$\mathcal{X}^{*}\to \Delta^{*}$, the way of 
compactifying as relative dlt minimal model $\X\to \Delta$ has 
a complicated indeterminancy caused by flops on $\X$. 
For the convenience of readers, 
we enhance the indeterminancy at polarized level as follows, 
thus clarifying the problem. Or see \cite{Kawamataflop}. 
Further, as we confirmed in \cite[beginning of \S4]{ops}, 
even in polarized setting, 
flops can easily cause {\it un-}separatedness of moduli, without polystability conditions. 

Nevertheless, let us first define the following notion of {\it weak K-moduli} 
compactifications, just to coin the well-studied notion 
in our context. 
We set the scene by fixing a connected Deligne-Mumford (at least generically smooth) 
moduli stack $\mathcal{M}$ 
of polarized log terminal Calabi-Yau varieties. 
We make a caution that in general even connected 
$\mathcal{M}$ could be a priori only of {\it locally} finite type 
(i.e., non-quasi-compact), 
due to unsolved problem on the boundedness of singularities 
(cf., \cite[1.1 (iii), 1.2]{YZha}, \cite[\S 9]{OO}) 
hence we simply take 
$\mathcal{M}$ which satisfies the quasi-compactness. 
Then, from the work of \cite{Vie}, 
$\mathcal{M}$ has 
a quasi-projective coarse moduli variety 
$M$. 

Note that from the definition, there is a universal famly $\pi\colon (\mathcal{U},\mathcal{L})\to \mathcal{M}$ 
of the polarized log-terminal Calabi-Yau varieties. 
Suppose $\mathcal{M}^{o}\subset \mathcal{M}$ denotes the open substack 
which parametrizes {\it smooth} polarizd Calabi-Yau varieties. 

\begin{Def}[Weak K-moduli and very weak K-moduli]\label{weak.Kmoduli}
For the above setting, we call the following object a {\it weak K-moduli stack} (resp., very weak K-moduli); 

A proper Deligne-Mumford moduli stack $\overline{\mathcal{M}}$ compactifying $\mathcal{M}$ (resp., $\overline{\mathcal{M}}$ 
compactifying $\mathcal{M}^{o}$), with a $\Q$-Gorenstein family of 
polarized semi-log-canonical, or equivalently, K-semistable\footnote{the equivalence follows  
from \cite{Od, Od0} (also cf., \cite[\S 4]{Od15})}) 
Calabi-Yau varieties 
$\bar{\pi}\colon (\bar{\mathcal{U}},\bar{\mathcal{L}})\to \overline{\mathcal{M}}$ which 
extends 
$\pi\colon (\mathcal{U},\mathcal{L})\to \mathcal{M}$ (resp., the 
restriction of 
$\pi\colon (\mathcal{U},\mathcal{L})\to \mathcal{M}$ to 
$\mathcal{M}^{o}$). 
For $\mathcal{M}$ to be weak K-moduli (but not for very weak K-moduli), we also require that 
underlying family of varieties 
$\bar{\mathcal{U}}\to \overline{\mathcal{M}}$ 
 is effective (i.e., no isomorphic varieties occur as fibers at 
 different $k$-rational points).  
\end{Def}

The notion is clearly motivated by the examples as \cite{AN, Ale02, Nak} (cf., also \cite{Nam, HKW} preceding that) 
for when $M=A_g$ the moduli of $g$-dimensional  
principally polarized abelian varieties, and that of (special) polarized K3 surfaces by \cite{Shah, Looi.semitoric, GHKS, Laza2, Inc, LO, AET, AB, ABE} among others. Both are 
obtained by explicitly determining logarithmic generalization of the KSBA compactification by fixing some ``natural" ample divisors 
in the Calabi-Yau varieties. Also, it often turns out that the compactification or its normalization is 
obtained as toroidal compactification \cite{AMRT} 
or semi-toric compactification \cite{Looi.semitoric}. 
As {\it loc.cit} works mentioned above show, the following has been 
a classical problem, certainly well-recognized 
folklore expectations among experts over decades. 

\begin{Conj}[Weak K-moduli (folklore)]\label{weak.Kmoduli.conj} 
For any moduli algebraic stack\footnote{in the sense of Deligne-Mumford \cite{DM}}
 $\mathcal{M}$ of  polarized log terminal Calabi-Yau varieties with its quasi-projective coarse moduli variety $M$, 
there is at least one weak K-moduli proper stack. 
Moreover, if $\mathcal{M}$ has uniformization 
by a Hermitian symmetric domain, then at least one such 
compactification's normalization is dominated by one of toroidal compactifications \cite{AMRT} 
or semi-toric compactifications \cite{Looi.semitoric}. 
\end{Conj}
The reason the latter statements only predict domination is that, at least in a log version (with nonzero boundaries) or subdomain version i.e., 
when $M$ does not locally cover the Kuranishi space, there do exist examples when the Satake-Baily-Borel compactifications 
parametrize polarized semi-log-canonical log Calabi-Yau pairs (cf., e.g., \cite[\S 4.1]{dBS}, \cite[\S 7.2.1]{OO}). 
We believe the following 
certain weak form of the 
above conjecture \ref{weak.Kmoduli.conj} 
has been essentially known to experts. 
\begin{Thm}[cf., \cite{AK, BCHM, Fjn.ss, HX, KX2, ATW, Bir20}]
\label{veryweak.Kmoduli}
For given $\mathcal{M}^{o}$, at least one very weak K-moduli 
compactification exists. 
\end{Thm}
The proof 
follows from the two kinds of deep techniques, by combining them 
in a relatively simple manner; applying the 
weak semistable reduction \cite{AK, ATW} (or the classical 
semistable reduction \cite{KKMS} for one parameter setting) 
and then run the relative log minimal 
model program (established in \cite{BCHM, Fjn.ss, HX} etc) 
after that. Before these procedures, boundedness of the pairs in concern is needed, 
which is achieved in \cite{Bir20} at fairly broad setting including ours. 
We include the sketch proof of Theorem \ref{veryweak.Kmoduli} for convenience. 

\begin{proof}[Outlined proof of Theorem \ref{veryweak.Kmoduli}]
For the given $\mathcal{M}^{o}$ 
parametrising smooth varieties can be written as 
$\Gamma\backslash M'$ with $\Gamma$-equivariant 
flat projective family of polarized slc Calabi-Yau varieties 
which we denote as 
$(\mathcal{U},\mathcal{L}) \twoheadrightarrow \mathcal{M}
=[\Gamma\backslash M']$ or $(U,L)\twoheadrightarrow M'$ 
with $\Gamma$-action on it. 

Then we apply the functorial weak semistable reduction 
\cite{AK, ATW} to (an arbitrary compactification for base direction 
of) $U \to M'$ so that after replacing $M'$ by its (further) 
finite Galois cover, we can compactify as $M'\subset \overline{M'}$ 
with toroidal flat morphism $\overline{U}\to \overline{M'}$. 
Then run relative minimal model model 
over $\overline{M'}$ 
of $\overline{U}$ and consider a (absolutely) ample line bundle  
$\overline{L}$ which is relatively linear equivalent to $L$ 
over $M'$. Then, take a general relative section $D$ of $|m\overline{L}|$ for $m\gg 0$. Then take its relative 
lc model by \cite{BCHM, Fjn.ss, HX}. It becomes a desired flat family of slc Calabi-Yau varieties 
with ample divisor (class), extending $(U,mL)\twoheadrightarrow 
M'$. The semi-log-canonicity follows from applying the adjunction 
inductively. 
\end{proof}

Now, turning back to weak K-moduli, 
the mentioned explicit examples of 
(e.g., \cite{AN, Ale02, Nak, Shah, Looi.semitoric, Zhu, AET, ABE}) 
are all weak K-moduli, 
giving the affirmative confirmation of the full  conjecture~\ref{weak.Kmoduli.conj}, hence has been its supporting 
evidences. Most of the above are obtained as special case of 
{\it logarithmic version (i.e., a version 
with boundary divisors) of KSBA construction}, 
which is also discussed in recent \cite{Bir20, KX2}. 
However, we should keep in mind and emphasize 
that a priori these could be 
moduli of the pairs of varieties and their {\it divisors}, 
hence some locus could possibly preserve the ambient varieties 
while the divisors change. 

It turned out that these weak K-moduli compactifications of the moduli of (log-terminal) polarized Calabi-Yau varieties turned out to be 
at least {\it non-}unique. 
Indeed, 
there are at least two non-isomorphic examples of weak K-moduli compactification in the case of the moduli $\mathcal{F}_{2}$ of degree $2$ polarized K3 surfaces; 
one by \cite{Shah} (cf., also \cite{Looi.semitoric}), 
obtained as a simple Kirwan blow up 
of the GIT moduli of sextic curves, 
and the other more recent one by \cite{AET} 
with much higher Picard number. 
The moduli construction of \cite{AET} 
is in terms of the polyhedral decompositions of 
integral affine spheres (``tropical K3 surfaces") with singularities. The latter family of tropical K3 surfaces on the $19$-dimensional cone 
is rather close to the one constructed in 
\cite[\S 4]{OO} on $\mathcal{F}_{2}(l)$ in the sense they have same radiance obstructions, 
hence expected to be obtained from each other by moving worms generically. However, the affine structure itself is different as 
observed by V.Alexeev. We thank him for sharing the observation.

\vspace{5mm}
Now we go back to discuss our main contents, 
and set the notation and conventions before that. 

\subsection{Notation and Conventions}
In this paper, we work over an arbitrary algebraically closed 
field $k$ of characteristic $0$ unless otherwise stated, 
since often we use the minimal model program and 
resolution of singularities. 
Some parts which involve analytifications 
are discussed only over $\C$ (e.g., \S \ref{lmmod.sec}) 
or non-archimedean fields (e.g., \S \ref{appendixA}) 
as indicated, 
which means complete valuation fields with non-archimedean valuations. 
Many purely algebro-geometric arguments extend over more general 
fields but we omit specifications. 

In this paper, a Calabi-Yau variety or a K-trivial variety both interchangably means a 
{\it klt} projective variety with numerically 
trivial canonical class, and weak Ricci-flat K\"ahler metric for it (when $k=\C$) 
means the 
singular K\"ahler metric whose existence is established by 
\cite{Yau, EGZ.sing}. 
Polarization means ample $(\Q-)$line bundle and variety with polarization is 
said to be polarized, and whenever it is not genuine line bundle but 
only $\Q$-line bundle we clarify so. 
We often discuss over a pointed smooth $k$-curve germ $\Delta\ni 0$ or 
$\Delta^{*}:=\Delta\setminus 0$, 
but many of discussions go verbatim (or with slight refinements) 
to the complex disk $\{t\in \C\mid |t|<1\}$ (resp., 
$\{t\in \C\mid 0<|t|<1\}$) 
or ${\rm Spec}k[[t]]$ (resp., 
${\rm Spec}k((t))$). 


\subsection*{Acknowledgements}
The author thanks Sebastien Boucksom, 
Tetsushi Ito, Yang Li, Yoshiki Oshima, Song Sun for helpful comments on the earlier versions of the draft, which in particular 
improved the presentation. 
During this work, the author is partially supported by 
KAKENHI 18K13389 (Grant-in-Aid for Early-Career Scientists), 
KAKENHI 16H06335 (Grant-in-Aid for Scientific Research (S)), and 
KAKENHI 20H00112 (Grant-in-Aid for Scientific Research (A)) 
during this research. 


\section{Degenerations and their limit - Galaxy}\label{fiber.sec}

\subsection{Basic of dlt minimal models}\label{basic.dlt.sec}

This section discusses dlt minimal models which do not necessarily satisfy 
open K-polystability (of \cite{ops}), and their certain limits. In some part, 
we give technical refinements of \cite{MN, NX, KX, NXY}. 

\subsubsection*{Notation}\label{Notation}
In the following {\it dlt minimal model} (resp., 
{\it lc minimal model}) 
$\X\to \Delta$ 
means $(\X,\X_{0})$ is a dlt pair (resp., a lc pair), 
so that $\X_{0}$ is reduced in particular, 
and $K_{\X}+\X_{0}$ is relatively nef over $\Delta$, 
where $\Delta$ denotes either germ of pointed smooth curve or 
${\rm Spec}k[[t]]$. 
A slight difference with other references (e.g., \cite{NX}) is 
that we assume $\X_0$ is reduced, just for simplicity, which is not essential assumption because 
the semistable reduction theorem \cite{KKMS} combined with the semistable MMP \cite{Fjn} always allow to 
pass to this case without essential change (our following arguments also include explanation). 
Note that the relatively nefness implies relatively triviality 
for the case when the generic fiber has trivial canonical class. 
We note following points for the convenience, 
which are easy to prove from the definitions. 

Since our paper focuses on the Calabi-Yau varieties and 
their degenerations, unless otherwise stated, 
dlt minimal model also requires the general fiber to be 
klt Calabi-Yau varieties which are $\Q$-factorial. 
Note that in some references, dlt minimal model requires the 
total models to be $\Q$-factorial. However, unfortunately the conditions 
exclude various important examples 
e.g., well-studied Dwork-Fermat family of K3 surfaces, 
as they are {\it non-$\Q$-factorial dlt minimal models}, 
we discuss in our generality. 

\begin{Thm}[Basic of dlt models]\label{dlt.lem}
\begin{enumerate}
\item \label{dlt.dom}
Consider any small birational map 
$\varphi\colon X\to X'$ and a $\Q$-divisor $D=\sum_{i} a_{i}D_{i}$ 
on $X$ 
with $D':=\varphi_{*}D$, such that 
$(X,D)$ and hence $(X',D')$ are both log pairs 
i.e., $K_{X}+D$ and $K_{X'}+D'$ are $\Q$-Cartier. 
Suppose $(X,D)$ is dlt and satisfies the following assumption. 
\begin{quote}
For any strata $Z$ of $\lfloor D \rfloor$ i.e., 
connected component of $\cap_{i\in I}D_{i}$ 
with the index set $I$ satisfying $a_{i}=1$ for $i\in I$, 
$\varphi$ restricts to $Z$ still gives a birational map. 
\end{quote}
Then, $(X',D')$ is also dlt. 
(For converse direction, see below \eqref{lc.min}). 
\item \label{dlt.term}
If $(\X,\X_{0})\to \Delta$ is a dlt minimal model, 
then $\X$ is terminal. 

\item (cf., \cite{Kawamataflop},\cite[3.2.5]{NX}) \label{dual.cpx.same}
Take two dlt minimal models $(\X_{i},\X_{i,0}) (i=1,2)$ over $\Delta$ 
which coincide over $\Delta\setminus \{0\}$. Then they are 
connected by flops and the dual intersection complex 
of their central fibers $\X_{i,0}$ are canonically homeomorphic. 

\item \label{lc.min}
For any lc minimal model $(\X,\X_{0})\to \Delta$ and 
its log crepant dlt minimal model $(\X^{\rm dlt},\X^{\rm dlt}_{0})$ 
which dominates $\X$, the exceptional locus of $\X^{\rm dlt}\to \X$ 
is a union of some irreducible components of the central fiber 
$\X^{\rm dlt}_{0}$ and hence cannot be small unless isomorphism. 

Moreover, if further $(\X,\X_{0})$ is dlt, 
then $\X^{\rm dlt}\to \X$ is small and satisfies 
the assumption of \eqref{dlt.dom}. 
\end{enumerate}
\end{Thm}
As an example of \eqref{dlt.dom}, the well-studied 
degeneration of quartic into four hyperplanes 
is a dlt (non-$\Q$-factorial) model. 
Note that the claim does not respect $\Q$-factoriality. 
\begin{proof}
The claim \eqref{dlt.dom} can be shown as follows. 
From the definition, $(X',D')$ is also obviously log canonical, 
at least. Then, 
since log canonical centers of 
$(X',D')$ are the strata of $\lfloor D\rfloor$ 
(\cite[\S 3.9]{Fjn}). Therefore, the assumption implies the 
assertion. 

For proving the latter \eqref{dlt.term}, 
we take a non-snc locus $W$ of $(\X,\X_{0})$. 
Then for any prime divisor over $\X$ with the center inside $W\subset 
\X_{0}$, the 
central fiber takes the multiplicity of $\X_{0}$ as at least $1$, 
since $\X_{0}$ is Cartier (indeed principal). Hence the claim. 

The statements of \eqref{dual.cpx.same} were known (cf., \cite[3.2.5]{NX}) 
at least in terms of skeleta in the Berkovich analytification but an easy 
(re)proof of \eqref{dual.cpx.same} easily follows now. 
By above \eqref{dlt.term}, \cite{Kawamataflop} implies 
that $\X_{i}$ are connected by flops (log flips, in reality). Because of 
\cite[3.9]{Fjn}, we can restrict to simple normal crossing part (in particular, toroidal) 
of $\X_{i}$ to discuss dual intersection complexes. 
Then, from the toric basics, the toric flops only change polyhedral decomposition 
(cf., \cite{GS}). \footnote{although not every flop preserves the dual complex such as 
the classical example of Atiyah flop over cone over quadric 
(cf., \cite[Ex 17]{dFKX}). Indeed, in {\it loc.cit}, 
the divisor in concern is only {\it a part} of toric divisor so that the setup is  
different.}

The last item can be proven as follows. 
As in the Step1 of the proof of Theorem ~\ref{lc.center.bir}, 
any of the log canonical centers $W$ of $(\X^{\rm dlt},\X^{\rm dlt}_{0})$ 
which are not irreducible components of $\X^{\rm dlt}_{0}$ 
(i.e., if ${\rm codim}_{\X^{\rm dlt}}(W)\ge 2$) 
cannot be contracted since semi-log-canonicity 
of $\X_{0}$ implies $S_{2}$ condition. On the other hand, 
the prime divisor of $\X^{\rm dlt}$ contracted in $\X$ is 
automatically log canonical centers, hence the claim. 
The rest of the proof again follows from the $S_{2}$ condition of 
$\X_{0}$. 
\end{proof}

Recall that since the work of Gross-Siebert and Kontsevich-Soibelman 
on geometric understanding of Strominger-Yau-Zaslow mirror symmetry conjectures (cf., 
\cite{KS, GS} etc), 
the dual intersection complexes of the central fibers of 
the relative minimal models for Calabi-Yau fibration over curves has been 
focused for a long time. Particular attention has been put on the case of maximal degenerations and in that case, the dual complex 
has been expected to be the sphere $S^{n}$ (cf., \cite{GW, GS} etc). 

Partial progress of understanding homeomorphic type are discussed in 
\cite{NX, KX}. 
\cite[4.1.4]{NX}, which is the consequence of some results in 
\cite{KK, Kol11}, proves that it is a closed 
manifold at least in (real) codimension $1$. 
\cite{KX} makes a progress, proving the expectation of dual complex to be $S^{n}$ in $n=3,4$ under certain strictness condition on the Calabi-Yau varieties. 
Thanks to the inductive structure of the stratification, 
it is also confirmed in \cite{KX} 
that the concerned dual intersection complex is 
topologically a manifold up to dimension $4$ (and $5$ if the degeneration is 
simple normal crossing). 


\subsection{Minimizing property of 
dlt model and polarized version}
\label{un.min.sec}

As an easy uniqueness statement of the Calabi-Yau filling, 
we recall the following in a global projective setting, 
which provides an algebro-geometric background to 
non-collapsing limits of Ricci-flat K\"ahler manifolds with bounded diameters. 

\begin{Thm}[{\cite[\S 4, discussion after 4.1, 4.2(i)]{Od12}}]
\label{CY.minimization}
 Suppose $(\X,\mathcal{L})\to C \ni 0$ 
 is a flat projective polarized family, 
 which are 
 $n$-dimensional slc Calabi-Yau projective varieties over 
 $C\setminus \{0\}$ (set $\X^{*}$ as preimage over $C\setminus \{0\}$)
 i.e., $K_{\X^{*}/(C\setminus \{0\})}\equiv_{/(C\setminus \{0\})} 0$. 
We consider all possible birational transforms 
$(\X',\mathcal{L}')$ 
 along the 
fibers over (the preimages of) $0$, allowing base changes of $C\ni 0$ as $b\colon C'\to C$. 

Then the normalized degree of CM line bundle (\cite{FS,PT,FR}) 
$$\dfrac{{\rm deg}(\lambda_{\rm CM}(\X',\mathcal{L}'))}
{{\rm deg}(b)}
=\dfrac{(\mathcal{L}'^{\cdot n}.K_{\X'/C'})}{{\rm deg}(b)}$$ 
attains the minimum 
(resp., the minimum as the only minimizer) 
among all 
$$\dfrac{{\rm deg}(\lambda_{\rm CM}(\X',\mathcal{L}'))}{{\rm deg}(b)}$$ 
if and only if $\X'_{0}$ is semi-log-canonical (resp., klt) and 
$K_{\X'_{0}}\equiv 0$. 
\end{Thm}

Recall that the CM line bundle becomes a  positive multiple of 
the Hodge line bundle in case of Calabi-Yau family. 
As we discussed in {\it loc.cit}, 
\footnote{However, the author mistakenly wrote the content of 
(ii) also in the place of 
(i) in {\it op.cit}, by which the author wanted to mean above. 
We recently noticed the mistake but the arithmetic variant's proof 
had also been repeated in \cite[Theorem 3.4(i)]{Od.ari}. 
We apologize for the possible confusion it caused. 
} 
basically the proof of above Theorem~\ref{CY.minimization} 
is same as \cite{Od} 
in the isotrivial case. 

\begin{Cor}[{\cite[Cor 4.3]{Od12}, \cite{Bou}}]
If a punctured family of polarized log terminal Calabi-Yau varieties 
$(\X^{*},\mathcal{L}^{*})\to \Delta^{*}$ 
can be completed to $(\mathcal{X},\mathcal{L})\to \Delta$ 
with a log terminal Calabi-Yau filling $(\mathcal{X}_{0},
\mathcal{L}_{0})$, 
there is no other slc Calabi-Yau filling.
\end{Cor}
In particular, moduli space of 
log terminal polarized Calabi-Yau varieties will be automatically 
separated (Hausdorff).

Contrast to the above uniqueness claim, 
as an existence-direction claim, we have the following. 
The result slightly refines the above written case, 
while ensuring ampleness of the extended polarization, 
and generalizes the classical work of Shepherd-Barron 
\cite{SB.K3} for degeneration of surfaces. 

\begin{Prop}[Polarized dlt model]\label{pol.dlt.model}
Given a punctured family of polarized log terminal Calabi-Yau varieties 
$(\X^{*},\mathcal{L}^{*})\to \Delta^{*}$, 
there is a filled-in proper flat family 
$(\X,\mathcal{L})\to \Delta$ such that 
$(\X,\X_{0, red})$ is dlt and $\mathcal{L}$ is 
{\bf relatively ample}. Here, 
$\X_{0, red}$ means the reduced scheme structure on 
$\X_{0}$. 

Furthermore, after finite base change, 
one can assume $\X_{0}$ is reduced i.e., 
$\X$ is dlt minimal model with ample $\mathcal{L}$. 
\end{Prop}

\begin{proof}
We take a dlt $\Q$-factorial minimal model $\X'$ and 
extend the polarization $\mathcal{L}^{*}$ to 
$\Q$-line bundle $\mathcal{L}'$ on $\X'$. Then, take a 
general section of $|m\mathcal{L}'|$ and denote as $\mathcal{D}'$. 
We take a relative log canonical model of 
$(\X',\epsilon \mathcal{D}')$ for $0<\epsilon\ll 1$ 
which gives the desired property. 
We denote its total space by $\X_{lc}$ and the intermediate relative 
log 
(dlt) minimal model's total space as $\X_{min}$. 

Note $\X_{lc}$ 
is still terminal hence Cohen-Macaulay from 
Theorem \ref{dlt.lem} \ref{dlt.term}. Thus its central fiber is 
again Cohen-Macaulay so that the generic reducedness implies 
reducedness. The only remained thing to confirm, for the 
former statement of this theorem \ref{pol.dlt.model} is that 
none of lc centers of the $\X_{min}$ is not contracted 
by the morphism to $\X_{lc}$. This also follows from the 
Cohen-Macaulay property of $\X_{lc,0}$. 

After finite base change, the semistable reduction theorem 
implies one can ensure reducedness of the central fiber of $\X'$. 
After passing to the relative lc model, the central fiber is still 
reduced because of the dlt property. 
\end{proof}

We introduce  
a simpler local version of the CM degree minimization 
(cf., \cite[\S 4]{Od12}). 
Although our main interest lies in the Calabi-Yau case for 
a moment, we discuss in more general setting. 

\begin{Def}[(Log) canonical height] 
For a proper flat ($\Q$-Gorenstein) family of 
polarized projective varieties over a smooth 
proper curve 
$(\X,\mathcal{L})\to C$ 
and a holomorpihc section $s\colon C\to \X$, we define  
the {\it log canonical height} 
(resp., {\it canonical divisor height}) 
of $s$ as 
$$h_{\rm lc}(s):={\rm deg}_{C} \hspace{1mm} s^{*}(\mathcal{O}_{\X}(K_{\X/C}+\X_{0,{\rm red}})).$$
(resp., $h_{\rm c}(s):=
{\rm deg}_{C} \hspace{1mm} s^{*}(\mathcal{O}_{\X}(K_{\X/C})).$)
Here, $\X_{0,{\rm red}}$ denotes the 
reduced divisor fully supported on $\X_{0}$. 
\end{Def}

This is a variant of weight function by \cite{KS, MN, Tem.diff}, 
and the result below characterizes its minimization easily, 
in a similar manner to {\it loc.cit}. 

\begin{Prop}[Dlt minimal models and heights minimization]
\label{minimization}
For a generically log terminal 
proper family of projective varieties 
$\X^{*}\to (C\setminus \{0\})$, and fixing 
the meromorphic section $s|_{C\setminus 0}$, 
consider fillings i.e., 
proper flat $\X\to C$ 
possibly after a finite base change of $C$ of degree $d$. 

Then, among such fillings with $s(0)$ lying in only 
one irreducible component of $\X_{0}$, the normalized lc divisor height 
$$\frac{h_{lc}(s)}{d}$$ is minimized 
if and only if $s(0)$ lies in an irreducible component 
which appears in the dlt minimal model of a finite base change 
(we only see local germ around $s(0)\in \X$). 
\end{Prop}

\begin{proof}
By a finite normalized base change (semistable reduction), 
which we denote as $\X'$, we can and do assume $s(0)$ lies in 
the reduced locus of $\X'_{0}$. 
Take a dlt minimal model $(\X,\X_{0,{\rm red}})$. 
We can take a resolution of indeterminancy of 
$\X'\dashrightarrow \X$ to reduce to the case when 
$\X'$ dominates $\X$ by a morphism $\varphi$. Then since 
$(\X,\X_{0,{\rm red}})$ is dlt by the definition, 
$K_{\X'}+\X'_{0,{\rm red}}-\varphi^{*}(K_{\X}+\X_{0,{\rm red}})$ 
is effective. Furthermore, since $\X$ is terminal, 
all log canonical valuations $v$ of $(\X,\X_{0,{\rm red}})$ 
have $v(t)>1$ but after a base change with 
the ramification degree $v(t)$ at $0\in C$, 
the log canonical valuations are realized at the dlt model 
of the base change over $C'$. 
Hence the both assertions easily follow. 
\end{proof}

We call the above irreducible component, 
{\it lc height minimizing} for $s$.


\subsection{Subdivision by base change and Galaxies}\label{subdiv.sec}
From now, we consider finite ramified 
base change of degenerating Calabi-Yau varieties 
and its ``subdividing'' effects, following \cite[\S4]{ops}. 
More precisely, we obtain a projective system of reductions as 
follows. Here, one main point will be to work over the (formal) 
Puiseux power series ring which means 
$$k[[t^{\Q}]]:=\cup_{m\ge 1}k[[t^{\frac{1}{m}}]],$$
in this paper. We also denote its fraction field 
as $k((t^{\Q}))$. 

\begin{Def}
\begin{enumerate}
\item \label{adm.dom}
Consider a (polarized) 
dlt minimal model $(\X,\mathcal{L})\to \Delta\ni 0$, 
and another dlt minimal model $(\X',\mathcal{L}')\to \Delta'\ni 0$ 
over a finite ramified cover of $\Delta$ whose 
complements of the central fiber is identified after 
the base change. 
Note that $\varphi$ is automatically birational and log crepant. 
We call $(\X',\mathcal{L}')$ (resp., $\X'$) is 
{\it admissibly dominating} $(\X,\mathcal{L})$ 
(resp., $\X$) 
if 
$\varphi\colon \X'\dashrightarrow 
\X\times_{\Delta}\Delta'$ is a morphism (i.e., everywhere defined) 
such that the following conditions hold (resp., the second condition holds): 
\begin{itemize}
\item 
$\mathcal{L}'$ and $\varphi^{*}\mathcal{L}$ 
coincide at open strata $\X_{0}^{\rm klt}$, 
\item $\varphi$ is toroidal at the snc (open) locus of $(\X,\X_{0})$. 
\end{itemize}

\item Take a local uniformizer $t$ of $\mathcal{O}_{\Delta,0}$ 
and consider $\mathcal{O}_{\Delta,0}\hookrightarrow 
k[[t]]\hookrightarrow k[[t^{\Q}]],$ where the last target denotes  
the local ring of formal Puiseux power series. 
If $(\X',\mathcal{L}')\to \Delta'$ admissibly dominates 
$(\X,\mathcal{L})\to \Delta$ in the above sense \eqref{adm.dom}, 
we call $$(\X',\mathcal{L}')\times_{\Delta'} k[[t^{\Q}]]
\to (\X,\mathcal{L})\times_{\Delta} k[[t^{\Q}]]$$ 
or its restriction to the central fiber 
$(\X'_{0},\mathcal{L}'_{0})\to(\X,\mathcal{L})$ 
also admissibly dominating. 
\end{enumerate}
\end{Def}
The term {\it admissibly} comes with reminiscence to 
that of Raynaud \cite{Ray}, although our main concern is 
more detailed analysis of very particular case of admissible blow ups in 
the sense of {\it loc.cit}. 

The base change trick in \cite[\S 4.1]{ops}, whose origin 
at least 
goes back to \cite{FriS} for Kulikov degenerations of 
K3 surfaces and Kodaira for $I_{\nu}$-type degenerations of 
elliptic curve case 
(cf., also \cite{KKMS}, 
\cite[\S 5.9]{BJ16} etc)  shows the following 
proposition. 

\begin{Prop}\label{subdiv}
For any dlt model $(\mathcal{X},\mathcal{L})\to \Delta$ 
as polarized Calabi-Yau family, and its ramified finite 
base covering $\Delta'\to \Delta$, we have a 
dlt model 
$(\mathcal{X}',\mathcal{X}'_{0})$ as a vertical blow up of 
$\mathcal{X}\times_{\Delta} \Delta'$ 
admissibly dominating 
$\mathcal{X}$. 
\end{Prop}

\begin{proof}
This is essentially proven in \cite[the proof of Theorem 4.8]{ops}, 
especially Step 2 and Step 3 in {\it loc.cit}. So we refer to {\it loc.cit} 
and omit the repetition of the arguments, except for pointing out that 
we should take the regular subdivision to ensure the 
relative projectivity criterion by Tai (\cite[Chapter IV, \S2, 
esp., Theorem 2.1]{AMRT}). 
Note that from the construction, the obtained 
dlt model is $\Q$-factorial even if the original $\X$ is 
not $\Q$-factorial. 
\end{proof}

The above construction of admissibly dominating models does {\it not} 
give their uniqueness, especially for the ambiguity 
caused by the simple use of the relative minimal 
model program. 
Here is basic properties of the admissible dominations. 

\begin{Thm}\label{subdiv.lc}
For an admissibly dominating morphisms between dlt minimal models as above 
$\psi\colon \X'\to \mathcal{X}\times_{\Delta} \Delta'$, 
the following set of closed subsets of $\X_{0}$ are the same: 
\begin{enumerate}
\item 
the set of log canonical centers of $(\X,\X_{0})$, 
\item the set of components of intersections of some components of $\X_{0}$, 
\item  the set of the images of log canonical centers of $(\X',\X'_{0})$ by $\psi$, and 
\item the set of components of intersections of some components of $\X'_{0}$.
\end{enumerate}
\end{Thm}
\begin{proof}
Before describing the proof, we set further notations: 
we denote the ramification degree of $\Delta'\to \Delta$ at the origin as $N$, 
$\mathcal{X}\times_{\Delta} \Delta'$ as $\X^{(N)}$. We take a log 
resolution $\varphi\colon \widetilde{\X}\to \X$ with a simple normal crossing vertical 
divisor $\DD$ of $\tilde{\X}$ 
so that $(\tilde{\X},\DD)$ is log crepant to $(\X,\X_{0})$, 
which exists by the definition of log resolution. 
Here, in this paper, 
the notion of simple normal crossingness for a divisor means locally normal crossing and 
 any finite intersection of 
the components are connected and smooth. 

Note that the divisor $\DD$ can be written as $\varphi^{-1}_{*}\X_{0}+E$ where 
$E$ is supported on $\varphi$-exceptional locus. 
Now, we set  $\widetilde{\X}^{(N)}:=\widetilde{\X}\times_{\Delta}\Delta'$ with its 
vertical divisor $\widetilde{\DD}^{(N)}$ such that the log smooth pair 
$(\widetilde{\X}^{(N)},\widetilde{\DD}^{(N)})$ is log crepant to $(\widetilde{\X},\mathcal{D})$ 
and $(\X,\X_{0})$. 

From the simple normal crossing property of $\DD$, 
$(\widetilde{\X}^{(N)},\widetilde{\DD}^{(N)})$ is not only log canonical, but even toroidal 
due to the local equations of two terms. Furthermore, 
since $\DD$ is supposed to have only connected smooth finite intersections, 
we can take a toroidal log resolution of the pair 
$(\widetilde{\X}^{(N)},\widetilde{\DD}^{(N)})$ as 
$(\widetilde{\widetilde{\X}}^{(N)},\widetilde{\widetilde{\DD}}^{(N)})$. 

In this situation, we confirm the following desired 
natural identifications of the sets. 
\begin{align}
\label{1} &{\text{log canonical centers of }}(\X^{(N)},\X^{(N)}_{0})\\
\label{2} =&{\text{log canonical centers of }}(\X^{(N)},\DD^{(N)})\\ 
\label{3}=&{\text{strata (components of finite intersections) of }}\lfloor \widetilde{\DD^{(N)}} \rfloor\\
\label{4}=&{\text{strata (components of finite intersections) of }}\lfloor \widetilde{\DD} \rfloor \\ 
\label{5}=&{\text{log canonical centers of }}(\X,\DD). 
\end{align}
The reasons of the above to hold are: 
\eqref{2} follows from the log crepantness, 
the \eqref{3} follows from toroidal structures, 
\eqref{4} obviously holds, and \eqref{5} follows from the 
dlt-ness of $(\widetilde{\X},\DD)$ thanks to 
\cite[3.9]{Fjn} for instance. 

Now, for the setting of the statements, 
we can take $\widetilde{\widetilde{\X}}^{(N)}$ 
as which dominates $\X'$. Then, sending the elements (loci of various varieties) 
to $\X_{0}$, we obtain the desired identifications. 
\end{proof}

It readily follows from 
above Theorems \ref{dlt.lem}, \ref{subdiv.lc} that: 

\begin{Cor}
In the setup of \ref{subdiv}, \ref{subdiv.lc}, 
\begin{enumerate}
\item \label{flop.cor}
For fixed $\X\to \Delta$ and $N$, $\# \{{\text{components of }\X'_{0}}\}$ 
does not depend on the choice of dlt minimal model $\X'$. 

\item \label{compo.asym}
$$\dfrac{\# \{{\text{components of }\X'_{0}\}}}{\# \{{\text{components of }\X_{0}\}}}\sim N^{m},$$
for $N\to \infty$, 
where $m$ denotes the dimension of the dual intersection complex. 
\end{enumerate}
\end{Cor}
\begin{proof}
The proof of \eqref{flop.cor}: 
it immediately follows from Theorem\ref{dlt.lem} \eqref{dlt.term} 
and \cite{Kawamataflop}. 
More precisely, we have canonical identification of the sets 
for different choices of $\X'$ (for fixed $\X$ and $N$). 

The proof of \eqref{compo.asym}: 
Recall from \cite{Kol11} that the minimal log canonical centers have 
all the same dimensions, hence the dual 
intersection complex is union of $m$-dimensional simplices 
(\cite[4.1.4]{NX}). 

Therefore, since $\Delta(\X'_{0})$ is 
canonically homeomorphic to $\Delta(\X_{0})$ with $N$ multiplied 
affine structures (Lemma~\ref{dual.cpx.same}), we conclude the proof. 
\end{proof}
\noindent
Note that before taking the limit, above can not be an equality 
unless abelian varieties case. 

Recall that for dlt model $\X\to \Delta$, with 
the generic fiber $\X_{\eta}$ and the central fiber 
$\X_{0}$, 
we have the essential skeleton 
$\Delta^{alg}(\X_{0})\subset \X_{\eta}^{an}$ as 
introduced in \cite{KS}, with more refined understanding 
in \cite{MN, BFJ, BFJ.sol, NX}. 
As \cite[\S 4.1]{KS} inferred indirectly and written explicitly 
in \cite[\S 3.2]{MN} (cf., also \cite[\S 2.1]{BFJ.sol}), 
it also admits a natural $\Z$-affine structure. 
Recall from the basic of Berkovich geometry \cite{Ber90}, 
for the degree $N$ ramifying finite map $\Delta'\to \Delta$ 
and $\X^{(N)}:=\X\times_{\Delta}\Delta'$ as before, 
there is a natural map 
$(\X'_{\eta})^{an}/{\rm Gal}(\Delta'/\Delta)\simeq 
\X_{\eta}^{an}$. 

\begin{Prop}
In the setup of Theorem~\ref{subdiv.lc}, 
the restriction of homeomorphism: 
$$(\X'_{\eta})^{an}/{\rm Gal}(\Delta'/\Delta)\simeq 
\X_{\eta}^{an}$$
gives a natural homeomorphism between 
the dual intersection complexes of $\X'_{0}$ and 
$\X_{0}$ as 
$$
\Delta^{alg}(\X'_{0}) \simeq  
\Delta^{alg}(\X_{0}), 
$$
with the $\Z$-affine structure 
simply $N$ multiplied (in the sense of 
local affine coordinates). 
\end{Prop}

\begin{proof}
If we replace $\X'$ by another dlt minimal model 
(without changing $\Delta'\to \Delta$), the essential skeleta 
do not change including the affine structures, as proven in 
\cite[\S 3.2]{NX} after \cite{MN}. 

Therefore, we can assume that $\Delta'$ is the one constructed 
in Step 2 of the proof of Theorem 4.4 in \cite{ops}, 
as used in \eqref{subdiv}. From the construction, 
over simple normal crossing locus of $(\X,\X_{0})$ 
which contains all lc centers, the birational morphism 
$\X'\to \X^{(N)}=\X\times_{\Delta}\Delta'$ is toroidal 
and the exceptional divisors corresponds to 
$1/N$-integral points of the maximal simplices of 
$\Delta^{alg}(\X_{0})$, from the definition of 
the affine structures in 
\cite[\S 3.2]{MN} (cf., also \cite[\S 2.1]{BFJ.sol}). 
Hence, from  Theorem \ref{subdiv.lc}, 
we conclude with the desired assertion. 
\end{proof}

\vspace{3mm}
Now we consider to pass the construction of Proposition~\ref{subdiv} over 
{\it Puiseux formal power series ring} over a field $k$ 
as $k[[t^{\Q}]]$ by base change. 
Then, 
the above admissible dominating morphism over such Puiseux formal 
power series ring is also a finite type blow up morphism. 
For simplicity of our noation, 
we denote the Puiseux formal power series ring over a field $k$ 
as $k[[t^{\Q}]]$, although it is {\it strictly smaller than} 
the whole set of 
formal power series with rational exponents as it may look. 
Note that although our base ${\rm Spec}(k[[t^{\Q}]])$ is 
non-Noetherian, the blow up still makes sense for any 
quasi-coherent ideal of finite type. 
General such blow up may not be ``fppf'' i.e., finitely presented a priori, 
but nevertheless in our case fppf holds because of the toroidal  description 
above and \cite[proof of Theorem 4.4]{ops} (compare with \cite[37 
0EV4]{Stpro}). 
Now, we are ready to introduce the notion of (quasi-)galaxies. 
\footnote{Intuitively (non-mathematically) speaking, galaxy in this sense (resp., galaxy in normal sense) looks like 
a totality which, while connected, consists of 
so many tiny fine pieces, each of which is still 
rich and beautiful world 
(log Calabi-Yau variety resp., planetary system). 
This might also remind some readers of dragonfly's compound eyes, or either 
Kumiko or Kiriko, the traditional crafts in Japan.}

\begin{Def}[(quasi-) Limit dlt model and (quasi-) Galaxies]\label{limit.defs}
\begin{enumerate}

\item \label{limit.model}
We call a locally ringed space $\X_{\infty}$ over $k[[t^{\Q}]]$ is 
{\it quasi-limit dlt model} of polarized Calabi-Yau family if there is 
a projective system of the base changes to $k[[t^{\Q}]]$ 
of dlt minimal models with admissbly dominating morphisms between them, 
whose projective limit is $\X_{\infty}\to k[[t^{\Q}]]$, 
which we call {\it quasi-galaxy model} (over $k[[t^{\Q}]]$). 
We call such projective system a presentation of $\X_{\infty}$ and its 
each dlt minimal model as {\it dlt approximation model}. 

We call the reduction (at $t=0$) of limit dlt model as 
{\it quasi-limit sdlt reduction} or {\it quasi-galaxy}. 
For simplicity, we often 
denote the reduction at $t=0$ of $\X_{i}$ (resp., $\X_{\infty}$) 
as $X_{i}$ (resp., $X_{\infty}$) from now on. 
Also, if $k$ is a valuation field, 
we put $X_{i}^{an}$ the complex analytification of 
$X_{i}$ and $X_{\infty}^{an}:=\varprojlim_{i}X_{i}^{an}$. 

\item \label{limit.ops}
A quasi-limit dlt model or its reduction i.e., quasi-limit sdlt reduction 
is said to be open K-polystable if one can take dlt approximation models sequence as 
all open K-polystable ones (cf., \cite{ops}). 

\item \label{complete}
We call a quasi-limit dlt model over Puiseux series in the above sense 
is {\it complete} or {\it limit dlt model} if it is 
represented by a sequence of dlt approximation models $\X_{i}$ over 
$k[[t^{1/N_{i}}]]$ satisfying: for any positive integer $N$, 
there is $i_{0}$ with $\forall i>i_{0}$, $N \mid N_{i}$. 
Intuitively, viewing the dual intersection complexes, 
this means ``sufficiently divided'' 
at its combinatorial looking. 
We call the reduction $X_{\infty}$ or $X_{\infty}^{an}$ (at $t=0$) of 
{\it limit dlt model} as {\it limit sdlt reduction} or simply {\it galaxy}, while the (quasi-galaxy) model $\X_{\infty}$ as 
{\it galaxy model} over $k[[t^{\Q}]]$. 

\item \label{limit.stable.model} 
We call a limit dlt model over Puiseux series in the above sense 
\eqref{limit.model} 
{\it limit (dlt K-)polystable model} if it is complete in 
the sense above \eqref{complete} and one can take 
the dlt approxmation models as all open K-polystable in the sense of 
\cite{ops}. 
\end{enumerate}
\end{Def}

\begin{Cor}\label{subdiv2}
For any dlt model $(\mathcal{X},\mathcal{L})\to \Delta$ 
as polarized Calabi-Yau family, there is a (complete) limit dlt model 
over $k[[t^{\Q}]]$ 
which dominates 
$(\mathcal{X},\mathcal{L})\times_{\Delta}k[[t^{\Q}]]$. 
\end{Cor}

\begin{proof}
This is a corollary to (the proof of) Proposition~\ref{subdiv} 
simply because $k[[t^{\Q}]]=\cup_{m}k[[t^{\frac{1}{m}}]]$ so that 
we can take the projective limit of their base changes. 
\end{proof}

\vspace{3mm}
Next statements give basic structure of galaxies, decomposing into pieces. 

\begin{Cor}[of Theorem~\ref{subdiv.lc}: Decomposition of Galaxies]\label{cor.subdiv.lc}
For any complete limit sdlt reduction $\X_{\infty,0}=\varprojlim_{i} \X_{i,0}$, 
it has a natural decomposition into open part and the closed part: 
\begin{align}\label{decomp.lim}
\X_{\infty,0}=(\sqcup_{a\in B(\Q)} \{\text{open klt log Calabi-Yau variety } U(a)\})
\bigsqcup \X_{\infty,0}^{NKLT}.
\end{align}
Here, $B$ denotes the dual intersection 
complex of $\X_{i,0}$ and $B(\Q)$ means 
its rational points with respect to 
the $\Q$-affine structure, which does not depend on $i$ by 
Theorem \ref{dlt.lem} \eqref{dual.cpx.same} and 
$\X_{\infty,0}^{NKLT}:=\varprojlim_{i}\X_{i,0}^{NKLT}$, 
where NKLT stands for the non-klt loci. 
\end{Cor}

\vspace{2mm}
Note that from Theorem \ref{dlt.lem} (iii), 
it follows that for each fixed $a$, 
birational type of $U(a)$ is unique. 
This Corollary \ref{cor.subdiv.lc} above 
is an immediate consequence of Theorem \ref{subdiv.lc}. 
We call the former part {\it klt locus} while the latter part 
{\it non-klt locus}, and denote them as $\X_{\infty,0}^{klt}$ 
(resp., $\X_{\infty,0}^{nklt}$) or $X_{\infty}^{klt}$ 
(resp., $X_{\infty}^{nklt}$) accordingly. 

As in \cite[\S 4]{ops}, from differential geometric perspective, 
an interesting case is maximal degenerations with open K-polystable components of the reduction. 
At least if the general fibers are abelian varieties, 
such open K-polystable reduction exists. 
Indeed, there is the N\'eron model after finite base change, 
whose abelian parts of the reduction 
are parametrized as the limits inside the Satake(-Baily-Borel) 
compactification 
of $A_{g}$ (cf., \cite[4.4.1]{Chai}). 
This is also confirmed in \cite[Cor 2.14]{TGC.II}. 
On the other hand, from the construction of 
\cite{FC90, Gub, Gub10} combined with the Delaunay decomposition 
used in \cite{Nam, AN}, we can relatively compactify to 
projective family. Hence, 
the existence of limit sdlt open K-polystable 
reduction does hold. 

\vspace{2mm}

Here comes the relation with non-archimedean geometry of Zariski-Riemann, Fujiwara, Huber type. 
We leave the details and the proof to later section \S \ref{ZR.review}. 

\begin{Prop}[Galaxies dominated by Huber analytification]\label{Huber.galaxy}
In the above setup, consider the generic fiber of $\X_{\infty}$ which we denote as $\X_{\infty, \eta}$. 
Then, any of its quasi-galaxy is dominated by the Huber analytification of the $k((t^{\Q}))$-variety 
$\X_{\infty, \eta}$ by a natural continuos surjective map. 
\end{Prop}


\subsection{Limit toric variety - a local toy model}\label{lim.toric.sec}

For (quasi-)limit sdlt reduction, 
often the irreducible components get blown up for infinitely 
many times. However, the blow up is certainly log crepant and 
not arbitrary. 
In this subsection, we introduce a prototypical example of 
or a toy model for such structure. Toys give us joy. 
This is prototypical not only 
to some constructions in the previous subsection but also \S 
\ref{Kmoduli.sec} for the compactified moduli. 

To make longer story of this subsection short, we consider the projective limit of 
toric variety. This may be of own interest in much more general 
context. To set the scene, 
recall for a lattice $N\simeq \Z^{n}$, 
proper toric variety $T_{N}{\rm emb}\Sigma$ of dimension $n$ 
corresponds to a rational polyhedral decomposition of $N_{\R}=N\otimes \R$. 

Here is the definition of limit toric spaces. 
Note that projective limit exists in the category of locally ringed spaces 
(cf., e.g., \cite[\S 4.1]{Fuj95}, \cite{Gillam}, \cite[Remark 2.1.1]{Tem.ZR}).  

\begin{Def}
For natural number $n$, 
we define locally ringed spaces 
$$\overline{T}^{n}_{\infty}:=\varprojlim_{\Sigma}T_{N}{\rm emb}\Sigma,$$ 
and if $k$ is a valuation field, we also set the 
Berkovich analytification version (\cite{Ber90}): 
$$\overline{T}^{n, an}_{\infty}:=
\varprojlim_{\Sigma}(T_{N}{\rm emb}\Sigma)^{an}.$$ 
Here, the directed set $\{\Sigma\}$ is defined as the set of all 
complete rational polyhedral decomposition of $N_{\R}$ 
and $\Sigma'\succ \Sigma$ if and only if $\Sigma'$ is subdivision of $\Sigma$. 
We call above locally ringed spaces $n$-dimensional {\it limit toric space}. 
\end{Def}

Note that this story is totally different from more ubiquitous 
toric construction for infinite, but locally finite, fan 
such as the one used in \cite{AMRT, Mum72.AV} among others. 
Here are some basic properties. 

\begin{Thm}\label{lim.toric.prop}
Suppose our base field $k$ is a valuation field. 
\begin{enumerate}
\item $\overline{T}^{n, an}_{\infty}\supset 
T_{N}^{an} \sqcup (\sqcup_{l} D_{l}^{o})$. 
Here, $T_{N}^{an}\simeq (\C^{*})\otimes_{\Z} N$ and 
$l$ runs over rational half lines inside $N_{\R}$. 
By $D_{l}$, we mean the torus invariant prime divisor of $T_{N}emb\Sigma$ 
with $\Sigma \ni l$, corresponding to $l$, and $D_{l}^{o}(\subset D_{l})$ denotes its 
torus invariant open subset. 
\item \label{lim.toric.conti}
If we denote the minimal Morgan-Shalen-Boucksom-Jonsson  compactification of $T_{N}^{an}$ (cf., \cite[Appendix A.1, esp., 
A.12]{TGC.II} and our Appendix \ref{appendixA}) as 
$\overline{T_{N}^{an}}^{MSBJ}$, the boundary is naturally isomorphic to 
$(N_{\R}\setminus \{0\})/\R_{>0}$. Further we have a natural surjective 
continuous map: 
$$\varphi_{tr}\colon 
\overline{T}^{n, an}_{\infty}\twoheadrightarrow \overline{T_{N}^{an}}^{MSBJ}.$$
Here, tr of $\varphi_{tr}$ stands for tropical. 
Moreover, at the boundary level, we have an algebraic version of the 
morphism $\varphi_{tr,alg}\colon 
\partial \overline{T}_{\infty}^{n}\to 
\partial \overline{T_{N}}$ which is also continous. 
\item \label{lim.toric.fiber}
Take a point $x\in (N_{\R}\setminus \{0\})/\R_{>0}=\partial 
\overline{T_{N}^{an}}^{MSBJ}$ (cf., \cite[Appendix A.1, esp., 
A.12]{TGC.II}, and our Appendix \S \ref{lim.toric.sec}). 
If $x=(x_{1},\cdots,x_{n})$ with respect to a basis of $N$ over $\Z$, 
and $r={\rm rank}_{\Q}\sum_{i} \Q x_{i}$, 
$$\varphi_{tr}^{-1}(\varphi_{tr}(x))\simeq 
\overline{T}^{n-r, an}_{\infty}.$$
\end{enumerate}
\end{Thm}

\begin{Ex}
Even for the case $n=1$, the structure of $\overline{T}^{1}_{\infty}$ is already somewhat 
interesting. For each nonzero $x\in N\otimes \Q$, there are two points 
$[x+0]$ (resp., $[x-0]$) in $\partial \overline{T}^{1}_{\infty}$ 
which values at $D_{l}\setminus D_{l}^{o}$ for any $\Sigma$ which contains $l$. 
This is analogous to Type $5$ point of \cite[2.20]{Scholze}. 
\end{Ex}

As we see, this limit toric space, as well as our analogues, 
has some mixed properties of varieties or usual analytic spaces, 
and adic spaces such as Zariski-Riemann spaces 
\cite{Zar, Hub94, Fuj95}. 

\begin{proof}
(i): We only need to 
consider $\Sigma$ whose ray set $\Sigma(1)$ includes $l$. 
Then, $T_{N}emb\Sigma$ contains $T_{N}\sqcup D_{l}^{o}$ as an 
open subset. This locus remains the same for any $\Sigma$ while further subdivisions, 
and for different $l$, $D_{l}^{o}$ does not intersect. Hence the proof. 

(ii): The map $\varphi_{tr}$ is constructed as follows. 
Take an arbitrary $$x=(x_{\Sigma})\in \overline{T}^{n, an}_{\infty}=
\varprojlim_{\Sigma}(T_{N}{\rm emb}\Sigma)^{an}.$$ 
For each $\Sigma$, we take $\sigma_{\alpha(\Sigma)}\in \Sigma$ 
such that its corresponding torus orbit $orb(\sigma_{\alpha(\Sigma)})
\subset T_{N}emb\Sigma$ contains $x_{\Sigma}$. 
Then, it follows that $\Sigma'\succ \Sigma$ implies 
$\sigma_{\alpha(\Sigma)}\supset 
\sigma_{\alpha(\Sigma')}.$ Since $\Sigma$ runs over all rational 
subdivisions, there is some $y\in N_{\R}\setminus \vec{0}$ such that 
$\R_{>0} y=\cap_{\Sigma} \sigma_{\alpha(\Sigma)}$. 
We put $[\R_{>0}y]\in \partial \overline{T_{N}^{an}}^{MSBJ}$ as $\varphi_{tr}(x)$. 

Then the continuity of $\varphi_{tr}$ can be confirmed as follows. 
We denote the natural projection map 
$\overline{T}^{n, an}_{\infty}\to T_{N}emb\Sigma$ as $p_{\Sigma}$. 
Now we take a point $y=\varphi_{tr}(x)\in \partial \overline{T_{N}^{an}}^{MSBJ}$. 
For an arbitrary open neighborhood $V$ of $y$, we want to show that there is an 
open neighborhood $U$ of $x$ so that $\varphi_{tr}(U)\subset V$. 
We can take such $U$ as follows. First we take a small enough 
regular rational polyhedral cone $\sigma$ whose image contains $y$. 
Then take a regular rational polyhedral decomposition of $N_{\Q}$ as 
$\Sigma$. Then one can consider the complex analytification $U_{\sigma}^{an}$ 
of affine toric variety $U_{\sigma}$ and set 
$$U:=(T_{N}^{an} \cup U_{\sigma}^{an}).$$ 
Then, it is immediate from the definition of $\varphi_{tr}$ 
that $\varphi_{tr}(U)\subset V$. Thus the continuity of $\varphi_{tr}$ 
is shown, as well as the corresponding map from 
the (algebraic) boundary since $U$ is the 
analytification of a 
Zariski open subset. 

(iii): After an appropriate element of $GL(N)\simeq GL(n,\Z)$, 
we can assume those vectors 
$(x_{1},\cdots,x_{r},x_{r+1},\cdots,x_{n})$ which satisfy that 
$x_{1},\cdots,x_{r}\in \R$ are rationally independent and 
$x_{i}$ for $i>r$ are all $\Q$-linear combination of $x_{1},\cdots,x_{r}$. 
Then, the assertion follows from direct explicit confirmation, 
which we leave to the readers. \end{proof}

The former (i) is analogous to Corollary \ref{cor.subdiv.lc} and 
the latter (ii) is analogous to Theorem \ref{tombo.tropical} to come. 
There is another further analogue in the context of moduli compactification, 
in \S \ref{Kmoduli.sec}. 


\vspace{3mm}
Now we go back to the original context of limit sdlt reduction 
at the end of \S \ref{subdiv.sec}. 

\subsection{Relation with the Kontsevich-Soibelman essential skeleta}\label{rel.skel}
Generally, we can show the following connetion of the 
limit sdlt reduction with 
essential skeleta proposed by Kontsevich-Soibelman 
\cite{KS} and more clarified in \cite{MN, NX} 
in the terminology of recent minimal model program. 
We first note the following, which e.g., 
follows from our construction above Proposition 
\ref{subdiv} 
(and \cite[proof of Theorem 4.4]{ops}). 

Note that for toric degeneration in the sense of 
\cite{GS}, we can also make sense of the essential skeleton or 
the dual intersection complex. In particular, 
such construction can be applied to 
any of the historical degenerations for abelian varieties, 
constructed by 
\cite{Mum72.AV, FC90, AN, Gub}: which we call 
semi-toric degeneration and their complete limit as 
{\it limit semi-toric degeneration/reduction}. 

\begin{Lem}\label{dual.cpx.same2}
Consider a dlt minimal model $\X$ 
over $k[[t]]$ and another dlt minimal model 
$\X'$ 
over $k[[t^{1/N}]]$ which is isomorphic to 
the base change of $\X$ away from $t=0$. 
Then the essential skeleta $\Delta(\X_{0})$ of $\X$ and 
$\Delta(\X'_{0})$ of $\X'$ are 
canonically homeomorphic, 
preserving the rational affine structures. 
\end{Lem}

Now we observe the limit sdlt reduction 
hides the structure of essential skeleton, 
by two versions of the statements: 
the case of abelian varieties with slightly general 
degenerations, and for general Calabi-Yau varieties case. 
Below, for abelian varietiese case, 
we allow {\it limit semi-toric degenerations} of abelian varieties 
which 
means the projective limit of the central fibers of 
approximation models constructed by \cite{Mum72.AV, FC90} 
with respect to subdivisions. 
More precisely, their approximation models are the 
relatively complete model of Raynaud extension divided by the 
period group $Y\simeq \Z^{r}$. Note that a priori such models are 
only formal, not necessarily algebraic nor projective, but for 
given punctured family, there do exists such a projective model 
thaks to \cite{Nam, AN, Zhu}. 

\begin{Thm}\label{tombo.tropical}
For any limit semi-toric degeneration $\X$ of abelian varieties 
over $k[[t^{\Q}]]$, 
$X_{\infty}^{\rm an}$ has a natural continuous surjective 
map to $\Delta(\X_{0})$: 
$$f_{tr}\colon 
X_{\infty} \to \Delta(\X_{0}).$$
\end{Thm}

\begin{proof}
The core arguments of the proof closely follows that of Theorem \ref{lim.toric.prop} \eqref{lim.toric.conti}, 
thanks to the toric nature of the construction of \cite{Mum72.AV, FC90}. 
Take a projective system of {\it all} semi-toric degenerations constructed 
by \cite{Mum72.AV, FC90, Gub} and denote it as 
$$\{\X_{i}^{[N_{i}]}\to \Delta\xrightarrow{t\mapsto t^{N_i}} \Delta\}_{i}.$$ 
Here, the index directed set of $i$ is the set of pairs of a positive integer $N_{i}$ and 
a $N_{i}Y$-periodic regular fan (cone decomposition) of $Y_{\R}\oplus \R_{>0}$. 
The order $i<i'$ is defined to hold if and only if 
$N_{i}|N_{i'}$ and the corresponding fan to $i'$ is a subdivision 
of that for $i$. 

From the construction in {\it loc.cit}, 
the obtained approximating semi-toric degenerations are quotients by $N_{i}Y\simeq \Z^{r}$ of 
a toric variety fiber bundle determined by the $N_{i}Y$-periodic fan, 
so that the dual intersection complex of the reductions are canonically 
homeomorphic to the $r$-dimensional real $(Y\otimes_{\Z} \R)/Y$ (cf., also 
\cite{TGC.II}). 

Take any $x=(x_{i})_{i}\in \varprojlim_{i}(\X_{i}^{[N_{i}]})^{an}_{0}$. 
For each $i$, take a semitorus strata (locally closed subset) 
$Z_{i}$ which contains $x_{i}$, and consider a cone $\sigma_{N_{i}}(x)$ 
which corresponds to $Z_{i}$. This is determined only modulo $N_{i}Y$. 
Nevertheless, considering the semi-toroidal structure of the admissibly 
dominating morphisms, we can take $\{\sigma_{i}(x)\}_{i}$ so that 
$i<i'$ implies $\sigma_{i}(x)\supset \sigma_{i'}(x)$. 
Then, we consider $\cap_{i}\sigma_{i}(x)$. Because $i$ runs over 
the index set of all admissible regular $Y_{\Q}$-rational polyhedral cone decomposition, 
it easily follows that is a half line i.e., there is $\tilde{f}_{tr}(x)\in Y_{\R}$ so that 
$\R_{\ge 0}(1,\tilde{f}_{tr}(x))=\cap_{i}\sigma_{i}(x)$. 
Then we set $f_{tr}(x):=\tilde{f}_{tr}(x) {\rm mod. } Y$. 

Now we show the continuity of $f_{tr}$. 
We take an open neighborhood $V$ of $\tilde{f}_{tr}(x)$ inside $Y_{\R}$, 
which is a section of a rational regular cone $\sigma$ of $N\oplus \R_{z}$ at the 
hyperplane $z=1$. We take a 
partial compactification of the Raynaud extension correponding to $\sigma$, 
i.e., affine toric variety $U_{\sigma}$-fiber bundle over the abelian part of 
$\X_{0}$. Then consider its image after the quotient by $Y$, and denote it as $\mathcal{U}$. We consider a regular rational $Y$-admissible 
polyhedral fan $\Sigma$ which contains $\sigma$. Then we consider the 
corresponding model $\X_{\Sigma}$. 
We define $U:=\mathcal{U}\cap X_{\Sigma}|_{t=0}$. 
Then, from the definition of $f_{tr}$ and $U$, 
it obviously holds that $f_{tr}(U)\subset V$. We complete the proof. 
\end{proof}
Note that, if we only allow regular fans for the semi-toric degenerations, 
obviously the obtained models are dlt (snc) models so that the inverse limit is a 
limit sdlt reduction. 

\vspace{2mm}

Here is another version of the statements for more general 
Calabi-Yau varieties. 

\begin{Thm}\label{tombo.tropical2}
For any punctured family of Calabi-Yau varieties 
$\X^{*}\to \Delta^{*}$, and its 
any galaxy model $\X$ over $k[[t^{\Q}]]$, 
the central fiber (galaxy) 
$X_{\infty}^{\rm an}$ 
has a natural continuous surjective map to $\Delta(\X_{0})$: 
$$f_{tr}\colon 
X_{\infty} \to \Delta(\X_{0}).$$
\end{Thm}

\begin{proof}
The proof closely follows that of above Theorem \ref{tombo.tropical}. 
Thus we briefly describe the proof basically as repition of 
the similar ideas, while showing some subtle 
differences. 
 
We again take a projective system of the approximating dlt models 
constructed 
and denote it as 
$$\{\X_{i}^{[N_{i}]}\to \Delta\xrightarrow{t\mapsto t^{N_i}} \Delta\}_{i}.$$ 
Here, the index directed set of $i$ is the set of pairs of a positive integer $N_{i}$ and the dlt model after the base change of 
degree $N_{i}$. 
The order $i<i'$ is defined to hold if and only if 
$N_{i}|N_{i'}$ and the corresponding polyhedral decomposition 
of the essential skeleton for $i'$ is a subdivision 
of that for $i$. 

Take any $x=(x_{i})_{i}\in \varprojlim_{i}(\X_{i}^{[N_{i}]})^{an}_{0}$. 
For each $i$, take a  strata (locally closed) 
$Z_{i}$ of the lc stratification of $X_{i}$ 
which contains $x_{i}$, and consider a 
rational polyhedron $\bar{\sigma}_{i}(x)$ in the 
essential skeleton $\Delta(X_{i})$, 
which corresponds to $Z_{i}$.

Note that our index set of $i$ is ordered so that 
$i<i'$ implies $\bar{\sigma}_{i}(x)\supset \bar{\sigma}_{i'}(x)$. 
Then, we consider $\cap_{i}\bar{\sigma}_{i}(x)$, 
which we can easily show to be a single point and denote it as 
$f_{tr}(x)$.

Now we show the continuity of $f_{tr}$ 
similarly to the case of abelian varieties. 
We take a closed neighborhood $\bar{\sigma}$ 
of $\tilde{f}_{tr}(x)$, 
which is a rational simplex inside the essential skeleton, 
which is one regular piece of the simplicial complex 
$\Delta(X_{i_{0}})$, corresponding to 
the log canonical center $Z_{i_{0}}$. 
We define the star $S(Z_{i_{0}})$ of the lc center $Z_{i_{0}}$ 
as the union of the strata of the lc stratification of $X_{i_{0}}$ 
whose closure contain the generic point of $Z_{i_{0}}$. 

From the definition, it is a Zariski open subset of 
$X_{i_{0}}$. We denote the projection $X_{\infty}\to X_{i_{0}}$ as 
$p_{i_{0}}$. 
Now 
we define $p_{i_{0}}^{-1}(S(Z_{i_{0}}))$ as $U$. Then, 
from the definition of $f_{tr}$ and $U$, 
it obviously holds that $f_{tr}(U)\subset V=\bar{\sigma}$. 
We complete the proof. 
\end{proof}

\begin{Rem}
The above may somewhat look resembling the 
Berkovich retraction \cite{Ber99,Thuillier} (also called ``non-archimedean 
SYZ fibration'' as in \cite{NXY}) from the Berkovich analytification of 
Calabi-Yau varieties to its essential skeleton (see also \cite{KS}).  
However, note that version was {\it not} canonical and changes by flops of the models,  
while our map $f_{tr}$ above is canonically defined. 
\end{Rem}


\subsection{Visualizing non-archimedean Calabi-Yau metric 
by galaxy models}

A natural basic strategy towards understanding the structure of limit 
sdlt reductions 
is to compare $\X_{0}$ and $\X_{i,0}$ which are 
the central fibers at $t=0$ of 
dlt minimal models $\X\to \Delta$ and $\X_{i}\to \Delta'$, 
where the bases are connected by a finite morphism 
$\Delta'\to \Delta$ with ramifying degrees $N_{i}$. 
From the definition \ref{limit.defs}, 
we have an admissibly dominating morphism $\X_{i}\to \X\times_{\Delta}\Delta'$. 
What about the ``converse-direction''? 
We partially show that for given any dlt minimal model $\X$, 
$\X_{i}$ can be taken which look locally very similar, 
refining Proposition \ref{subdiv}.

\begin{Prop}[Local identification]\label{subdiv.local.same}
Take an arbitrary 
dlt minimal model $\X$ which is a toric degeneration (\cite{GS.logI, GS11}), 
so that in particular the central fiber $\X_{0}$ is a Gorenstein stable toric variety in the sense of \cite{Ale02}. 
 We denote its open subset where $(\X,\X_0)$ with its projection to $\Delta$ is toroidal 
 as $\X^{tor}\subset \X$ (cf., \cite{GS.logI} \cite[3.9]{Fjn}). 
 For any finite branched morphism $\Delta'\to \Delta$ with any ramifying degree $N$, 
 there is a dlt minimal model $\X'\to \Delta'$ with an admissibly dominating 
 morphism $\psi\colon \X'\to \X\times_{\Delta}\Delta'$ such that the following holds: 
 
 For any point $x'\in \psi^{-1}\X^{tor} (\subset \X')$, 
 there is a point $x\in\X_{0}\cap \psi^{-1}(x')$ such that 
 $(x\in \X)$ and $(x'\in \X')$ have isomorphic germs. 
\end{Prop}
\begin{proof}
As in \cite{GS.logI, GS11}, the degeneration corresponds to 
the dual intersection complex $B$ with its decomposition $\mathcal{P}$ into 
a union of polyhedra $\{P_{i}\}_{i}$ respecting the integral affine structures. 

If we keep $(B,\{P_{i}\}_{i})$, while multiplying the coordinates of $P_{i}$ by 
$N$, we obtain a natural toric degeneration over the open part 
$\X^{\rm tor}\subset \X$ and then further 
apply the construction of $\X^{(N)}$ after Proposition \ref{subdiv}
(\cite[Theorem 4.8 Step2]{ops}), which depends on regular subdivision of the dual intersection complex. 
Here, we take the regular subdivision by the standard one by dividing by $N-1$ hyperplanes for each facets direction, 
which subdivides each simplex by unit simplices, so that the local structure remains. 
Then, we run the relative 
minimal model program over toroidal degenerating locus do not affect our assertion. 
 (The difficulties overcome by Gross-Siebert is to even 
deal with non-toroidal part, while starting from the central fiber $\X_{0}$.)
The assertion easily follows from the construction from the locally identical 
toric description. 
\end{proof}

\vspace{2mm}
\subsubsection*{Degenerating abelian varieties case}
Now we consider the case of families 
principally polarized abelian varieties as it is the simplest instance. 
Construction of degenerating abelian varieties via 
toric methods and formal geometry is well-established: 
cf., \cite{Mum72.AV}
\footnote{Note that \cite{Mum72.AV} constructed 
semiabelian reduction by auxiliary ``relatively complete model'' 
which is neither unique nor canonical. \cite[\S 6]{Gub} extended 
over general non-archimedean fields. }, 
 \cite{FC90}, 
\cite[\S 6]{Gub} and special cases for 
canonically compactifying $A_{g}$ are also done in 
more details by \cite{Nam, AN} among others. 

Above Proposition~\ref{subdiv.local.same} of ``subdividing models with locally same structures'' is vividly observed in 
the case of abelian varieties in the context of describing non-archimedean 
canonical Chambert-Loir measures (\cite{CL, Gub, Gub10}). 

Indeed, for given semitoric polarized degeneration of abelian varieties  
$(\X,\mathcal{L})$, 
\cite[Prop 6.7]{Gub} more explicitly 
constructed the admissbly dominating morphisms of 
any semi-toric degenerations of abelian varieties, 
which we denote as $\X^{(N),std}\to \X\times_{\Delta}\Delta'$ 
where ``std'' stands for ``standard'' and $\Delta'\to \Delta$ is a 
degree $N$ finite cover with ramification just at the origin. Further, he shows 
the non-archimedean Calabi-Yau metric (cf., \cite{BFJ.sol}) is described in terms of such 
approximation models in e.g., \cite[1.3]{Gub10} (cf., also \cite{Liu.AV}). 
More precisely, he shows: 

\begin{Thm}[\cite{Gub, Gub10}]
For any dlt semitoric polarized degeneration of abelian varieties 
$(\X,\mathcal{L})\to \Delta$, 
there is a natural polarization $\mathcal{L}^{(N)}$ on $\X^{(N)}$ 
and 
the non-archimedean CY metric $|\cdot |_{\rm nAMA}$ 
for the $(\X_{\eta}^{an},\mathcal{L}_{\eta}^{an})$ is 
the limit of the model metrics induced by $\mathcal{L}^{(N)}$ where 
$N=l^{a}$ with $a\to \infty$, for any positive integer $l$. 
Here $\X_{\eta}$ is the generic fiber of $\X\to \Delta$ and 
$\mathcal{L}_{\eta}$ is the restriction of $\mathcal{L}$ to it. 
\end{Thm}
In the proof, the group scheme 
structure (multiplication maps) of abelian schemes over $\Delta\setminus \{0\}$ 
is effectively used which of course does not exist for other 
Calabi-Yau varieties case. 
Comparing with our terminology of this section, we could paraphrase the above: 

\begin{Cor}[of \cite{Gub, Gub10}]
For any polarized punctured family of abelian varieties 
$(\X^{*},\mathcal{L}^{*})\to \Delta^{*}=\Delta\setminus \{0\}$, and 
any positive integer $l>1$, 
there is a quasi-galaxy model $\X_{\infty}$ with 
dlt approximation models $\{\X_{i}\to \Delta'={\rm Spec}(k[[t^{1/l^{i}}]])\}
_{i=1,2,\cdots}$ and their $\Q$-line bundles $\mathcal{L}_{i}$ satisfying the 
followings: 
\begin{enumerate}
\item 
$\mathcal{L}_{i+1}=p_{i,i+1}^{*}\mathcal{L}_{i}(-E_{i+1})$. Here, 
$p_{i,i+1}\colon \X_{i+1}\to \X_{i}\times_{k[[t^{1/l^{i}}]]}k[[t^{1/l^{i+1}}]]$ 
is the admissibly dominating morphism and $E_{i+1}$ is a $\Q$-Cartier 
$p_{i,i+1}$-exceptional divisor. 
\item if we denote the non-archimedean Calabi-Yau metric (\cite{Gub10, Liu.AV, BFJ.sol}) 
as $|\cdot |_{nAMA}$, 
the model metric of $\mathcal{L}_{\eta}^{an}$ 
induced by $\mathcal{L}_{i}$ as $|\cdot|_{\mathcal{L}_{i}}$, 
$$|\cdot|_{nAMA}=\lim_{i\to \infty}|\cdot |_{\mathcal{L}_{i}}.$$
\end{enumerate}
\end{Cor}

Morally speaking, the above observation connects for abelian varieties case: 
\begin{itemize}
\item limit open K-polystable sdlt reduction (originally motivated by 
understanding of bubbles of {\it complex (K\"ahler)} Calabi-Yau metrics), and 
\item the {\it non-archimedean Calabi-Yau metric} (cf., \cite{Gub10, Liu.AV, BFJ.sol}). 
\end{itemize}

\begin{Rem}
It is easy to see that for any quasi-limit dlt minimal models and 
their polarizations, the CM line bundle can be defined by Deligne-pairing 
nevertheless of non-Noetherianess of 
$k[[t^{\Q}]]$, $k[[t^{\R}]]$ (cf., \cite[\S 1.1, \S 1.2]{SWZha.ht}). 
\end{Rem} 


\section{Limit toroidal compactification}\label{lmmod.sec}

Some particular features of the limit toroidal compactification and the  limit log minimal compactification of a normal quasi-projective variety $M$ 
(our particular concern is when $M$ is the moduli of polarized log-terminal Calabi-Yau varieties), both  
to be introduced in this section, are summarized as follows: 
\begin{itemize}
\item they are locally ringed spaces, although {\it not} algebraic 
varieties. 
\item 
They are 
dominated by the Zariski-Riemannian compactification 
(\S \ref{ZR.review}) 
which is ``much bigger''. 
\item Unlike the Zariski-Riemann compactification, 
limit toroidal (or log minimal) compactification 
certainly respects and reflects the ``{\it minimality}'' 
in the sense of the (log) minimal model program. 
\item For connected Shimura varieties case at least, the 
limit toroidal compactification dominates both the Satake-Baily-Borel compactification (variety)
and the Morgan-Shalen type compactification (not variety), 
i.e., there are continous surjections to them. 
Compare such fact with that no variety compactification dominates 
any of the 
Morgan-Shalen type comapctification. 
\end{itemize}

\subsection{Construction and expectation}\label{Kmoduli.sec}

We set the scene as follows. 
Consider a class of polarized smooth 
K-trivial varieties with connected moduli $M^{o}$, and 
all their $\Q$-polarized 
Calabi-Yau degenerations i.e., degenerated klt Calabi-Yau variety 
$X$ with $\Q$-line bundle $L$ in $\frac{1}{N}{\rm Pic}(X)
\subset {\rm Pic}(X)\otimes \Q$ (for some $N\in \Z_{>0}$ depending on $X$), 
and its coarse moduli algebraic space of them all $M$. 
\begin{Ass}\label{Ass1}
We suppose all of the following (mutually related) assumptions: 
\begin{enumerate}
\item \label{ass1} $N$ has a uniform upper bound for all $X$s,
\item \label{ass2} $M$ is 
a quasi-projective scheme over $k$, 
\item more strongly, $M$ is 
a quasi-projective normal variety over $k$
\end{enumerate}
\end{Ass}
A famous result of Viehweg \cite{Vie} implies 
the first assumption 
\eqref{ass1} 
implies the second assumption \eqref{ass2}, hence the difficulty we face now is a boundedness type problem. 
This issue was also raised in \cite[1.1, 1.2]{YZha}, \cite[\S 8, 
\S 9]{OO}. 

It has been known for decades that 
the assumption \ref{Ass1} (all) holds for abelian varieties 
or K3 surfaces with polarizations, although 
it is non-trivial in general (cf., e.g., \cite{Gross.obst}, 
\cite{Nam.obst}, 
\cite[\S 9, Step1]{OO}). Recently it is also confirmed for 
$\Q$-Gorenstein smoothable hyperK\"ahler varieties with 
polarizations in \cite[\S 8.3]{OO}. 

Under the above assumption \ref{Ass1}, we can take a finite Galois cover $M'$ which is a normal quasi-projective variety 
on which there is a polarized family in concern (such a cover exists 
for arbitrary $M$ cf., \cite{Vie}). 
We denote the Galois group as $\Gamma:={\rm Gal}(M'/M)$. 
Further, take a $\Gamma$-equivariant 
projective compactification $M'\subset \overline{M'}$ 
such that $\overline{M'}\setminus M'$ is a simple normal crossing divisor 
$D'$. We also set $\overline{M}:=\Gamma\backslash \overline{M'}$, 
its boundary divisor $D:=\overline{M}\setminus M$. 
We denote the log canonical 
model of $(\overline{M'},D')$ as $M'_{\rm lc}$. 
The branch $\Q$-divisor with the standard coefficients
\footnote{meaning the usual $1-\frac{1}{d}$ where $d$ is the ramifying degree. 
cf., e.g., \cite[3.4]{Ale1}} 
in $M$ is denoted by $D_{M}$.

We consider the following. 

\begin{Def}\label{lim.min.const}
If $k=\C$ and $M$ has a structure as a locally Hermitian symmetric space  
(e.g., the complex analytification of a 
connected Shimura variety), let us consider 
the projective limit of all of its toroidal 
compactifications as locally ringed spaces and 
call them the {\it limit toroidal compactifications}. 
\begin{align}
\overline{M}^{\rm tor, \infty}&:=\varprojlim_{\Sigma} \overline{M}^{\rm tor,\Sigma}.\\
{\text{and its analytification  }}\hspace{2mm} \overline{M}^{\rm tor, \infty, an}&:=\varprojlim_{\Sigma} \overline{M}^{\rm tor,\Sigma,an}.
\end{align}
We mean 
$\overline{M}^{\rm tor,\Sigma}$ 
to be the toroidal compactification with respect to the $\Gamma$-admissible collection of fans $\Sigma=\{\Sigma(F)\}$
 in the sense of \cite{AMRT} 
and the above 
$\overline{M}^{\rm tor,\Sigma,an}$ 
means their complex analytifications. 
Here, $F$ denotes the rational boundary component of the Satake-Baily-Borel 
compactification of $M$ and 
$\Sigma(F)$ is the fan $\{\sigma_{\alpha}^{F}\}_{F}$ of $C(F)$ in the notation of 
\cite{AMRT}. 
Here, we add more assumption: 
\begin{Ass}\label{Ass.log.gen.type}
$(\overline{M},D_{M}+D)$ is of log general type and has only log 
canonical singularities. 
\end{Ass}
The latter is always true but we also 
expect the former to be true in general. 
See \cite{Zuo, Deng} for related partial results. 
Under the assumption, we set 
\begin{align}
\overline{M}^{\rm min,\infty}&:=\varprojlim_{\overline{M'}^{log.min}} 
(\Gamma\backslash \overline{M'}^{\rm log.min}), \\ 
{\text{and its analytification  }}\hspace{2mm}
\overline{M}^{\rm min,\infty,an}&:=\varprojlim_{\overline{M'}^{log.min}} 
(\Gamma\backslash \overline{M'}^{\rm log.min, an})
\end{align}
where $\overline{M'}^{\rm log.min}$ 
run over all $\Gamma$-equivariant 
log dlt minimal models of $(\overline{M'},D')$ 
and $\overline{M'}^{\rm log.min, an}$ 
denotes their complex analytifications. 
The projective limits are taken in the category of 
locally ringed topological spaces and we call these 
the {\it limit log minimal compactifications}. 
\end{Def}
It is easy to see that following holds. 
Below, we continue to use the above notation. 

\begin{Prop}
For either of the above two compactifications 
$\overline{M}^{\rm min,\infty}$, the boundary 
$\partial \overline{M}^{\rm min,\infty}$ contains as an open part, 
which we denote as 
$\partial \overline{M}^{\rm min,\infty,o}$
the union of the 
plt locus of $(\overline{M},D_{M}+D)$, 
where $(\overline{M'},D')$ runs over all $\Gamma$-equivariant log minimal models. 
\end{Prop}
\begin{proof}
Indeed, since the negativity lemma readily implies 
the log minimal models are log crepant, the plt 
locus can not be blown up. 
\end{proof}

We denote the union of the above open locus 
$\partial \overline{M}^{\rm min,\infty,o}$ 
in the boundary and 
$M$ as $\overline{M}^{\rm min,\infty,o}$. 

\begin{Prop}
If the limit log minimal compactification $\overline{M}^{\rm min,\infty}$
 dominates (the analytification of) a proper 
 {\it variety} $M\subset \bar{M}$, 
 then there is one of log minimal compactifications 
$(\Gamma\backslash \overline{M'}^{\rm log.min, an})$, 
such that there is a dominant birational map 
whose image at least 
contains open subset of $\bar{M}$ outside codimension $2$ locus. 
\end{Prop}

\begin{proof}
For each prime boundary divisor $F$ of $\bar{M}\supset M$, 
from the construction of \ref{lim.min.const}, 
there is a log minimal compactification 
$\Gamma\backslash \overline{M'}(F)$ which contains 
$M$ and the generic point of $F$. We take such log minimal 
compactification  for each $F$, and take 
a refined log minimal compactification 
which dominates all such 
$\Gamma\backslash \overline{M'}(F)$. This completes the proof. 
\end{proof}

The moduli-theoritic meaning of this construction has been 
not clear for the moment, while our first speculation would be 
that $\partial \overline{M^{\rm tor,\infty}}^{o}$ could 
parametrize certain reductions of 
``limit (dlt open K-)polystable models'' to be defined in 
section\ref{fiber.sec}. 
A possibly hinting result we obtain is the following. 

\begin{Thm}\label{dom.TGC}
Suppose $M$ is a locally Hermitian symmetric space 
$\Gamma\backslash G/K$. 
Then, there is a natural continuous surjective map 
\begin{align}
\phi_{tr}\colon \overline{M^{\rm tor,\infty}}^{an}\to \overline{M}^{\rm MSBJ}
\end{align}
which extends the identity map on $M$. 

Furthermore, using the notation in \cite{AMRT} and \cite[\S2]{OO}, 
if we take a general point $x$ inside the open strata 
$C(F)/\R_{>0}\subset \partial 
\overline{M}^{\rm MSBJ}$ with a $0$-dimensional cusp $F$ of the 
Satake-Baily-Borel compactification $\overline{M}^{SBB}$ then 
the fiber $\phi_{tr}^{-1}(x)$ is the limit toric variety 
$\overline{T}^{n-r, an}_{\infty}$ (defined in \S \ref{lim.toric.sec}) 
where $r$ denotes the 
$\Q$-rank of $x$ in $U(F)$, 
where $U(F)$ is the center of the unipotent radical of the 
real maximal parabolic subgroup of $G$ fixing $F$ (stratawise). 
\end{Thm}

Here, $\overline{M}^{\rm MSBJ}$ denotes the minimal 
Morgan-Shalen-Boucksom-Jonsson compactification (\cite[Appendix]{TGC.II},\cite[\S 2]{OO}) i.e., 
the MSBJ compactification corresponding to  
{\it the toroidal compactifications} \cite{AMRT}, 
which do not depend on the combinatorial data i.e., 
the admissible collection of rational 
polyhedra. 

\begin{proof}
The proof is very similar to 
that of Theorem \ref{lim.toric.prop} \eqref{lim.toric.conti} (and 
Theorem~\ref{tombo.tropical}) 
and essentially follows from the same arguments. 
Indeed, we prove as follows. 
First, we take an arbitrary 
$x\in \overline{M^{\rm tor,\infty}}$ and consider the rational boundary strata $F$ 
of the Satake-Baily-Borel compactification of $M$ which contains the image of $x$, 
following the notation of \cite{AMRT}. 
For some regular decompositions $\{\Sigma(F)\}$, we denote the natural projection 
$\overline{M^{\rm tor,\infty}}\to \overline{M}^{\Sigma}$ as $p_{\Sigma}$. 
We consider an open neighborhood of $x$ of the form 
$p_{\Sigma}^{-1}(U)\subset \overline{M^{\rm tor,\infty}}$ 
and define $\phi_{tr}$ from it, which does not depend on the choice, 
and glue them together. 

We denote the strata of $\partial \overline{M}^{tor, \Sigma, an}$ 
containing the image $x_{\Sigma}$ of $x$ as $S_{x}$. Then take a small enough 
open neighborhood $U'$ of $x_{\Sigma}$ such that 
$U'\cap F'\neq \emptyset $ if and only if $\bar{F'}\supset S_{x}$. 
We replace $U$ by $U'\cap U$ which is possible by the basic property of 
the stratification. From here, we further use the notation in \cite{AMRT} 
without reviewing all. We just recall that $U(F)$ is the 
center of the unipotent radical of $N(F)$ which is a real maximal 
parabolic subgroup of the corresponding reductive Lie group to $M$, 
$U(F)_{\Z}=U(F)\cap \Gamma$, and $D$ denotes the covering Hermitian 
symmetric domain. 

Since the action of $(\Gamma\cap N(F))/U(F)_{\Z})$ acts 
properly discontinously on the quotient $(D/U(F)_{\Z})_{\Sigma(F)} (\subset (D(F)/U(F)_{\Z})_{\Sigma(F)})$, 
we can lift the point $x_{\Sigma}=p_{\Sigma}(x)$ to 
$(D/U(F)_{\Z})_{\Sigma(F)}$. 
Suppose $x_{\Sigma}$ lies in the 
toric strata corresponding to a cone $\sigma_{\alpha(\Sigma)}^{F}\in \Sigma(F)$. 
Then consider $\cap_{\Sigma}\sigma_{\alpha(\Sigma)}^{F}$ similarly to 
Theorem \ref{lim.toric.prop} \eqref{lim.toric.conti}. Completely similarly, 
we can show that it is a half line of the form $\R v_{x}$ for $v_{x}\in 
U(F)\setminus \{\vec{0}\}$. Then we set $\phi_{tr}(x):=\bar{v_{x}}$. 
The proof of the continuity of the obtained map $\phi_{tr}$ 
(resp., the fiber structure) 
is completely 
same as Theorem \ref{lim.toric.prop} \eqref{lim.toric.conti} 
(resp., Theorem \ref{lim.toric.prop} \eqref{lim.toric.fiber}), 
so we avoid the 
further essential repetition. 
\end{proof}

\subsubsection*{A weak analogue for the moduli of curves $M_{g}$}
If we try to search analogue of the limit log minimal compactification 
for the moduli of hyperbolic curves $M_{g}$, moduli of 
hyperbolic curves of genera $g>1$, one would naturally replace 
the set of all toroidal compactifications above by the 
single Deligne-Mumford compactification 
$M_{g}\subset \overline{M_{g}}^{\rm DM}$ 
since it is 
smooth lc model at staky level. Therefore, we do not 
obtain a similar compactification in the same manner but still there 
is another compactification which is analogous to some extent 
(compare with above Theorem~\ref{dom.TGC}): 

Amini-Nicolussi \cite{Amini20} recently constructed a 
compactification $M_{g}\subset M_{g}^{\rm hyb}$ 
on whose boundary they parametrize {\it metric complex} 
(\cite{AB}) 
with the ordered partition of the edge sets which they call 
``layor''s. 

\begin{Prop}\label{AN.interpret}
The compactification of $M_{g}\subset M_{g}^{\rm hyb}$ constructed by 
Amini-Nicolussi is the least common refinement 
in the sense of \cite[I.16.1, I.16.2]{BJi} of 
\begin{itemize}
\item 
the Deligne-Mumford compactification $M_{g}\subset \overline{M_{g}}^{\rm DM}$, 
\item 
the Morgan-Shalen-Boucksom-Jonsson compactification 
$$M_{g}\subset \overline{M_{g}}^{\rm MSBJ}
(\overline{\mathcal{M}_{g}}^{\rm DM})$$ for 
the Deligne-Mumford moduli stack $\overline{\mathcal{M}_{g}}^{\rm DM}$ 
(\cite[Appendix]{TGC.II}, see also \cite{TGC.I}). 
\end{itemize}
\end{Prop}
\begin{proof}
Recall that the well-known local structure of $\mathcal{M}_g$ around the boundary and universal curves over it, 
says in particular there is a natural one to one order (closure relation) reversing bijection between the set of strata of 
$\overline{M_g}^{\rm DM}$ and the strata of $M_{g}\subset \overline{M_{g}}^{\rm MSBJ}
(\overline{\mathcal{M}_{g}}^{\rm DM})$, which is the moduli of tropical curves (\cite{ACP}). 
Tracing the bijection, the desired assertion follows from the construction of \cite[\S 3.2]{Amini20} as follows.  
For a sequence of $M_g$ towards the boundary converging to 
$t$ in a boundary strata $D_{F}=D_{E_t}$ in 
$\overline{\mathcal{M}_g}\setminus \mathcal{M}_g$, which corresponds to 
the index set (of nodes), 
the set of possible limits in $M_{g}\subset \overline{M_{g}}^{\rm MSBJ}
(\overline{\mathcal{M}_{g}}^{\rm DM})$ is the closure of the open strata in the dual intersection complex of 
the stacky snc divisor $\overline{\mathcal{M}_g}\setminus \mathcal{M}_g$ from the construction \cite[\S 2]{BJ}, \cite[Appendix]{TGC.II}. 
It is nothing but the $\sqcup_{\pi\in \Pi(F)}\sigma_{\pi}^o$ in \cite[(3.10, 3.11)]{Amini20}. 
Therefore, the least common refinement of the two compactifications parametrize (smooth projective hyperbolic curves and) 
the metrized complex in the effective manner. 
Since
the $D_{\pi}^{o}=D_{F}^{o}$-part of {\it loc.cit} (3.10) is nothing but the parametrizes the limit stable curves, 
while its $\sigma_{\pi}^{o}$-part parametrizes the graph part of the metrized complexes as \cite{Amini20} shows and 
also our assertion follows. 

A caution is that to recall $M_{g}\subset \overline{M_{g}}^{\rm MSBJ}
(\overline{\mathcal{M}_{g}}^{\rm DM})$ is not compatible with hyperbolic metric behaviour, i.e., 
different from the  Gromov-Hausdorff compactification of $M_g$ with respect to 
hyperbolic metrics on the Riemann surfaces. Indeed, we needed to replace ``glueing function" to describe the Gromov-Hausdorff compactification. 
See \cite{TGC.I} and \cite{TGC.II} for details. 
\end{proof}

From the above proposition \ref{AN.interpret}, 
$M_{g}^{\rm hyb}$ and $\overline{M}^{\rm tor,\infty}$ for 
locally Hermitian symmetric space $M$ can be both understood as 
the least common refinements of 
\begin{itemize}
\item log minimal model compactifications (projective varieties), and 
\item its (their) Morgam-Shalen compactification(s). 
\end{itemize}
The critical difference of $M_{g}$ case and 
connected Shimura variety case is 
that, at stacky level, 
we can take unique log minimal model compactification 
(actually lc model) by \cite{DM} 
as the standard choice, while the latter admits 
many log minimal model compactifications as toroidal compactifications 
by \cite{AMRT}. Hence, $M_{g}^{\rm hyb}$ can be seen as a weak 
analogue of $\overline{M}^{\rm tor,\infty}$ and 
$\overline{M}^{\rm min,\infty}$ nevertheless of 
many differences. 


\subsection{Zariski-Riemann compactification and 
comparison}\label{ZR.review}

\subsubsection{Review of Zariski-Riemann compactification}

Here we review a Zariski-Riemann type compactification which 
dominates the above limit log minimal compactification. 
Although obviously the idea goes back to Zariski's innovative idea 
\cite{Zar}, the reason of our review is that simply we could not find 
literatures precisely mentioning the results of the following form. 
So we hope the following accounts worth writing as an appendix, 
and perhaps contain slight improvements of the known results. 

Recall that the dual intersection complex is 
tightly connected to the theory of Berkovich analytic space 
as its ``finite part''. Therefore,  it is natural to 
explore the possible great enrichment of the Morgan-Shalen type 
compactification through 
Zariski-Riemann space or Huber adic space, 
which is the aim of this section. 

Suppose $U$ is a $k$-variety. Recall that, as we declared in the introduction, 
$k$-variety for a field $k$ in this paper 
means integral separated finite type scheme over $k$. 
Here we introduce another ``canonical'' compactification of $U$ 
with the {\it Zariski topology} using the theory of 
Zariski-Riemann spaces \cite{Zar}. 

\begin{Def}In the above setup, we set 
$${\rm Val}_{k}(k(U))=\{\text{all (Krull) valuations } v\colon k(U)\to \Gamma\}/\sim,$$ 
where $k(U)$ denotes the meromorphic function field of $U$ as usual, 
$\Gamma$ runs over totally ordered abelian groups, 
$\sim$ denotes the equivalence of (Krull) valuations. 
Denote $$\partial \Val(U):=\{v\colon k(U)\to \Gamma\mid c(v)\notin U\},$$
where $c(v)$ denotes the center of $v$ in a projective compactification including $U$. 
We put the topology (Zariski topology) whose 
open basis are rational domains 
$U(x_{1},\cdots, x_{m})=\{v\mid v(x_{i})\ge 0 \text{ for }\forall i\}$. 
\end{Def}
On the other hand, here is another classical notion.
\begin{Def}[Zariski-Riemann space \cite{Zar40, Zar}]
In the above setup, we set 
$$ZR_{k}(k(U)):=\varprojlim_{X} X,$$
(in the category of locally ringed spaces) 
where $X$ runs over all proper varieties including $U$ as an open subset (we fix 
the inclusion), 
with birational proper morphisms between them to form a projective system. 
\end{Def}

We denote the center of $v$ in $X$ as $c_{X}(v)$ following a standard notation. 
Then, recall that the reduction (specialization) map gives a homeomorphism: 
\begin{Thm}[Zariski \cite{Zar}]\label{Zar.isom}
In the above setup, 
there is a natural homeomorphism 
$$Val_{k}(k(U))\simeq ZR_{k}(k(U)),$$
which sends $v$ to the centers $\{c_{X}(v)\}_{X}$. 
\end{Thm}

Below is a natural variant where $U$ is preserved as not blown up. 
\begin{Def}[{\cite{Zar40, Zar} also cf., \cite[\S 4]{Fuj95}}]\label{ZR.cptf2}
In the above setup, 
consider $$ZR(U):=\varprojlim_{U\subset \overline{U}} \overline{U},$$
where $\overline{U}\supset U$ is a 
proper variety which includes $U$ as an open dense subset. As before, we put the 
topology as (the restriction of) the product topology of the Zariski topology 
on each scheme $\overline{U}$. 
\end{Def}
Obviously, we have a natural decomposition 
$$ZR(U)=U\sqcup \varprojlim ((\overline{U}\setminus U)=:\partial{\overline{U}}),$$ 
so we set $\partial ZR(U):=\varprojlim_{\overline{U}} \partial{\overline{U}}$. Then we have:

\begin{Prop}[cf., \cite{Zar, Fuj95, Tem.ZR}]
In the above setup, 
\begin{enumerate}
\item \label{two.ZR.isom}
for the natural morphism 
$$p_{U}\colon ZR_{k}(k(U))\to ZR(U),$$
its restriction
$$p_{U}|_{p_{U}^{-1}(\partial ZR(U))}\colon p_{U}^{-1}(\partial ZR(U)) \to \partial ZR(U)$$ 
is isomorphism. 

\item \label{Var.ZR.isom}
There is a natural homeomorphism 
$$\partial Val(U)\simeq \partial ZR(U).$$
\end{enumerate}
\end{Prop}

\begin{proof}
\eqref{two.ZR.isom}: 
Consider a point in $ZR_{k}k(U))$ which can be identified as a Krull valuation on 
$k(U)$ by the classical theorem~\ref{Zar.isom}. We consider its image in 
$ZR(U)$, which we suppose to be outside $U$. We denote its germ as 
$\mathcal{O}$ with its maximal ideal $\mathfrak{m}$, 
while $(\mathcal{O}_{v},\mathfrak{m}_{v})$ is the valuation ring for $v$. 
Note that we have $\mathcal{O}\subset \mathcal{O}_{v}\subset k(U)$. 
What remains to show is $\mathcal{O}_{v}=\mathcal{O}$ as the rest automatically follows. 
We prove by contradiction, so suppose the contrary. 
Take 
\begin{align}\label{f/g}
\frac{f}{g}\in \mathcal{O}_{v}\setminus \mathcal{O},
\end{align}
where $f, g \in \mathcal{O}$. 
To make arguments (even) simpler, we prepare a reference compactification variety $U\subset 
\overline{U}_{ref}$ 
such that $U_{ref}\setminus U$ is a Cartier divisor. This is possible as 
otherwise we can take the blow up of $\partial \overline{U}=\overline{U}\setminus U$. 
We denote the germ of image of $v$ in $\overline{U}$ as $\mathcal{O}'(\subset \mathcal{O})$, 
around which the Cartier divisor $\overline{U}\setminus U$ of $\overline{U}$ 
is locally generated by a single element $h\in \mathcal{O}'$. 
Since $\mathcal{O}$ is a local ring, $g\in \mathfrak{m}$ holds. 
Suppose $f\notin \mathfrak{m}$. Then, it would imply $\frac{1}{g}\in \mathcal{O}_{v}
\setminus \mathcal{O}$ by multiplying $\frac{1}{f}$ to \eqref{f/g}. 
Combining $g\in\mathfrak{m}$ with $\frac{1}{g}\in \mathcal{O}_{v}$, 
$1\in \mathfrak{m}_{v}$ follows which is absurd. 
Hence it follows that $f\in \mathfrak{m}$. Then, we can take a blow up of 
$\overline{U}_{ref}$ along (some closure of) the locally closed subscheme 
cut by the ideal $(f,g,h)$. Then it easily contradicts to \eqref{f/g}.\\ 

\eqref{Var.ZR.isom}: 
The construction is just by restricting the map of more 
classical theorem \ref{Zar.isom}. 
First, we confirm the continuity as follows. 
We take an arbitrary reference compactificaion variety model 
$U_{ref}\supset U$ as before, and restrict to an 
open affine subset $U'$ of $U_{ref}$. Since the assertion is of local nature 
with respect to $U_{ref}$, it is enough to show the continuity 
on the preimage of the (arbitrarily taken) open affine subset $U'$ 
of $U_{ref}$. 
We take an arbitrary closed subscheme $V=V(\{f_{i}\}_{i})$ 
with finite index set of $i$, then consider the preimage 
of $U'$. The preimage is $\{v\mid v(f_{i})>0 {\text{ for }} \forall i\}$, 
which can be also written as the complement $(\cup\{v\mid v(f_{i}^{-1})\ge 0\})$. 
The latter is closed by definition. 

The proof of openess here is close to the arguments of \cite[3.4.6 (cf., also 3.1.2, 3.1.8, 3.4.7)]
{Tem.ZR}. We can again restrict the problem to open affine subset of $U_{ref}$, 
say $\Spec(A)$
Then the image of an arbitrary rational domain, which we denote as 
$U(\{\frac{f_{i}}{g_{i}}\})$ is preimage of a natural open subset of 
the blow up $Bl_{(\langle \{f_{i},g_{j}\}_{i,j}\rangle)}(\Spec(A))$. 
It is open by definition, hence the proof. 
\end{proof}

Now we temporarily go back to the discussions of galaxies and 
prove Proposition \ref{Huber.galaxy}, which we recall here. 
Note that for any non-archimedean field $K$ and finite type scheme $X$ over $K$, 
there is a naturally associated adic space $X^{ad}$ in the sense of \cite{Hub96}, consisting of semi-valuations of a priori arbitrary rank, 
which here we call the {\it Huber analytification} or {\it Huber adification}. 
See e.g., \cite[Definition 2.2.1]{Fos} (also cf., \cite[\S2]{Scholze}) for the explicit definition. 

\begin{Prop}[Galaxies dominated by Huber analytification ($=$Proposition \ref{Huber.galaxy})]\label{Huber.galaxy2}
For any dlt model $\X$ over $\Delta$ which we base change to $k[[t^{\Q}]]$ and denote it as $\X_{\infty}$. 
We denote its generic fiber as $\X_{\infty, \eta}$. 

Then, any of its quasi-galaxy $X_{\infty}$ is dominated by the Huber analytification of the $k((t^{\Q}))$-variety 
$\X_{\infty, \eta}$, which we denote as $\X_{\infty, \eta}^{ad}$, by a natural continuos surjective map: 
$$\X_{\infty, \eta}^{ad}\twoheadrightarrow X_{\infty}.$$
\end{Prop}

\begin{proof}[proof of Proposition \ref{Huber.galaxy}=\ref{Huber.galaxy2}]

The proof resembles that of more classical Theorem by Zariski \ref{Zar.isom}. 
Take the dlt approximation models $\X_i$ with their central fibers $X_i$s and consider their base change over $k[[t^{\Q}]]$. 
As in the classical setting of \ref{Zar.isom}, for any semi-valuation $v=\mid - \mid _v$ in $\X_{\infty, \eta}^{ad}$, 
we take all the centers $c_{X_i}(v) \in X_i$ for each $i$. This gives a continous map similarly to \eqref{Zar.isom} (cf., \cite{Zar}). 

On the other hand, the surjectivity also follows similarly as \cite{Zar}. 
We consider all the normal projective models of $X$ over $k[[t^{1/m}]]$, base changed to $k[[t^{\Q}]]$ and denote  the models as $\mathcal{Y}_{j}$s and 
their reductions as $Y_j$s. 
Then take the projective limit $\varprojlim_j Y_j$ 
in an analogous manner to Zariski-Riemann compactification as above \ref{ZR.cptf2}. 

Take a point inside $X_{\infty}$ which is the image of $\{y_j\}_j \in \varprojlim_j Y_j$. 
We consider the natural injective limit $\varinjlim \mathcal{O}_{Y_j, y_j}$, 
where $\mathcal{O}_{Y_j, y_j}$ denote the stalk local ring of $y_j \in Y_j$. 
We denote it as $\mathcal{O}$ and prove it is a valuation ring as desired. 
Suppose the contrary and take $r=\frac{f}{g}$ in the fraction field of $\mathcal{O}$, with $f, g \in \mathcal{O}$. 
We can suppose $f,g$ are both realised in the model $Y_j$ for same $j$ which we fix. Then, 
we take a blow up of $Y_j$ with respect to the ideal $(f,g,t)$. Then we obtain another model $Y_k$ 
and easily see that either $r$ or $r^{-1}$ is in the stalk of $\mathcal{O}_{Y_k}$ at $y_k$, hence the proof. 
\end{proof}

\subsubsection{Berkovich type compactification as separated quotient}

First we re-interpret the construction of \cite[\S 3]{PM} (also cf., \cite[\S 2, \S 4]{BJ16}, \cite[Appendix]{TGC.II}), which compactify (the complex analytification of) 
a complex variety $U$ with ``non-archimedean'' boundary which is normalized Berkovich space 
in the sense of \cite{Fan}. They denoted it as $U^{h}\subset U^{\neg}$. 
See \cite{PM} for details. 

We add some remarks to their work. For simplicity, we put smoothness assumption on $U$,  
but as in \cite{TGC.II} or our Appendix \S \ref{appendixA}, 
we should be able to relax the condition without difficulties. 
Take an arbitrary $\overline{U}$ which is a smooth proper variety containing $U$, 
with the complement a simple normal crossing divisor. 
Applying \cite[\S 2.2]{BJ}, we obtain 
$$\overline{U}^{MSBJ}(\overline{U})=U^{an}\sqcup \Delta^{alg}(\overline{U}\setminus U)$$ 
as well as 
$$\overline{U}^{lim.MSBJ}=\varprojlim_{\overline{U}} \overline{U}^{MSBJ}(\overline{U}),$$
as in \cite[\S 4.2, \S 4.4]{BJ} both as 
compact Hausdorff topological spaces. 
Here $\Delta^{alg}(-)$ means the (algebraic) dual intersection complex. 

\begin{Prop}
\begin{enumerate}
\item \label{Berk.MSBJ}
$\partial \overline{U}^{lim.MSBJ}=\varprojlim_{\overline{U}} \Delta^{alg}(\overline{U}\setminus U)$ 
can be naturally identified with the set of $\R$-valued valuation of 
$k(U)$ with the centers lie outside $U$, modulo multiplication of positive real numbers. 
In particular, it admits a projectivized affine structure in the sense of 
our appendix \S \ref{proj.str.sec}. 

\item \label{ZR.Berk}
A natural continuous surjective map 
$\tau\colon \partial ZR(U)\twoheadrightarrow \partial \overline{U}^{lim.MSBJ}$ exists as separated quotient. 
\end{enumerate}
\end{Prop}

\begin{proof}
\eqref{Berk.MSBJ}: 
This is a close analogue of \cite[Theorem 1.13]{BFJ.sing} and \cite[Corollary 2.5]{BFJ}. 
Indeed, essentially the same proof works so we leave the details to the readers. 

\eqref{ZR.Berk}: 
Since $\partial ZR(U)$ is a valuative topological space in the sense of \cite{FK}, 
we can consider a selfmap $\tau\colon \partial ZR(U)\to \partial ZR(U)$ by sending 
to the unique maximal generization, whose image is exactly the 
$\partial \overline{U}^{lim.MSBJ}$. Hence, what remains is to show the 
continuity. From the definition of the projective limit topology, 
induced by the product topology, it is enough to show the following: 
if we fix a compactification $U\subset \overline{U}$ of the above type i.e., 
with simple normal crossing $\overline{U}\setminus U$ and its local 
coordinate chart with the coordinate $f_{1},\cdots,f_{m}$ such that 
the local equation of $\overline{U}\setminus U$ is $f_{1}\cdots f_{m}=0$. 
Then there is a corresponding $m$-simplex inside $\Delta^{alg}(\overline{U}\setminus U)$. 
It is enough to show that for the limit set of 
$\{\frac{\log|f_{j}|}{\log|f_{i}|}\in (a,b)\}$ inside 
$\partial \overline{U}^{lim.MSBJ}$, which we denote as $S_{i,j}(a,b)$, 
the preimage $\tau^{-1}(S_{i,j}(a,b)) (\subset \partial ZR(U))$ 
is open. However, it is easy to see that 
$\tau^{-1}(S_{i,j}(a,b))$ can be written as 
$$\bigcup_{a<\frac{l_{2}}{l_{1}}<\frac{m_{2}}{m_{1}}<b}
(\left\{v \in \partial ZR(U) \middle|  \|\frac{f_{2}^{l_{1}}}{f_{1}^{l_{2}}}\|_{v}\le 1 \right\}
\cap 
\left\{v\in \partial ZR(U)  \middle|  \|\frac{f_{1}^{m_{1}}}{f_{2}^{m_{1}-m_{2}}}\|_{v}\le 1 \right\}),$$
where $l_{i}$ and $m_{i}$ are positive integers. 
The above subset of $\partial ZR(U)$ 
is open from the definition of the topology on $(\partial) ZR (U)$. 
\end{proof}

Note that in the above, we could only discuss boundaries.




\section{Family construction over tropical geometric compactifications}
\label{fiber.mod.sec}

So far, in this paper or previous \cite{OO18}, we have been discussing 
some canonical moduli 
compactifications and the degeneration relatively independently. 
This section can be morally 
seen a first step to connect them in a somewhat 
more direct manner. In this section, we assume $k=\C$. 

The main point is that 
the procedure of the Morgan-Shalen-Boucksom-Jonsson compactification is ``functorial'' 
with respect to morphisms as proven in 
\cite[A.15]{Od.Ag}. The Morgan-Shalen-Boucksom-Jonsson compactification was 
introduced only for certain varieties and their compactification varieties (cf., 
\cite{MS84, BJ}, \cite[Appendix]{TGC.II}), 
but our appendix \S \ref{appendixA} 
generalizes it to much more general analytic spaces, for both 
complex and Berkovich non-archimedean analytic spaces. 
Below, first we essentially review \cite[A.15]{Od.Ag}. 
\footnote{However note that in the proof of \cite[A.15]{Od.Ag}, 
mistakenly, the reasoning of $k\ge l$ is not written in 
a sufficient manner. The correct reasoning is obtained 
by injectively mapping the set of the local boundary divisors downstair, labeled by $\{1,\cdots,l\}$, into the set of 
local boundary prime divisors upstair, labelled by $\{1,\cdots,k\}$ by looking at the pullback of generic points of the divisors. }

\begin{Prop}[Functoriality, algebraic case]\label{functoriality.alg}
Let $f\colon X\to Y$ be a $\C$-morphism between normal complex varieties, $D_{X}$ be a divisor on $X$, 
$D_{Y}$ on $Y$, both satisfying the assumption \ref{assumption.alg}. Here we further assume \cite[Def 8 (2)]{dFKX} 
i.e., each intersection of irreducible components of the boundary is irreducible. Then there is a natural map $\Delta^{\rm alg}(D_{X})\to \Delta^{\rm alg}(D_{Y}).$ 
\end{Prop}

\begin{proof}
We take a Zariski open covering of $X$ (resp., $Y$) denoted as $\{U_{i}\}_{i}$ (resp., $V_{i}$) where $U_{i}$ maps to $V_{i}$ for fixed $i$ and 
each $U_{i}$ or $V_{i}$ contains only one stratum of biggest codimension in it. We do so, by first taking $V_i$s which contains 
just one smallest dimension stratum and take $U_{i}$ as 
subdivisions of the preimages of $V_{i}$s (while the 
subdivision of $U_{i}$, we duplicate $V_{i}$s for different 
subindices $i$). 

Then $\Delta^{\rm alg}(D_{X}\cap U_{i})$ and $\Delta^{\rm alg}(D_{Y}\cap V_{i})$ 
are both simplices. Note that for any $i,j$, $\Delta^{\rm alg}(D_{X}\cap U_{i}\cap U_{j})$ is a closed subcomplex of both 
$\Delta^{\rm alg}(D_{X}\cap U_{i})$ and $\Delta^{\rm alg}(D_{X}\cap U_{j})$ because the strata inside $\Delta^{\rm alg}(D_{X}\cap U_{i}\cap U_{j})$ is 
closed under specialization. The same for $V_{i}$s. 
As discussed in \cite[A.15]{Od.Ag}, $f$ induces a natural simplicial map from 
$\Delta^{\rm alg}(D_{X}\cap U_{i})$ to $\Delta^{\rm alg}(D_{Y}\cap V_{i})$, which glues at the closed subcomplex 
$\Delta^{\rm alg}(D_{X}\cap U_{i}\cap U_{j})$s. Therefore, we obtain a continuous stratified map from 
$\Delta^{\rm alg}(D_{X})$ to $\Delta^{\rm alg}(D_{Y})$. 
\end{proof}

Now, let us observe the following examples of  
degenerations of affine 
curves over higher dimensional base: 

\begin{Ex}\label{Atiyah.flop}
An easy example is a small crepant resolution of $V(xy-st)\subset \A^2_{x,y}\times \A^2_{s,t}$ with the natural projection $f$ down 
to $\mathbb{A}_{s,t}^{2}$. 
Recall that there are two such crepant small resolutions, say $\mathcal{X}_i$ for $i=1,2$ both dominated by the blow up at $(x,y,s,t)=(0,0,0,0)$ of the quadric cone 
$V(xy-st)=:\mathcal{Y}$ and are connected by the well-known Atiyah flop. Here we can assume $\mathcal{X}_{1}$ is 
the blow up of $\mathcal{Y}$ along $(x,z)$ (or 
equivalently, that along $(y,w)$) and 
$\mathcal{X}_{2}$ is the blow up of $\mathcal{Y}$ along 
$(y,z)$ (or equivalently that along $(x,w)$). 

We set $D_{\mathcal{X}_i}:=V(st)\subset \mathcal{X}_i$ and $U_{\mathcal{X}_i}:=
\mathcal{X}_i\setminus D_{\mathcal{X}_i}$ 
Then Proposition~\ref{functoriality} applies to yield $(U_{\mathcal{X}_i}^{\rm an})^{\rm MSBJ}(\mathcal{X}_i)\to 
((\mathbb{A}_{s,t}^2\setminus (st=0))^{\rm an})^{\rm MSBJ}(\mathbb{A}_{s,t}^{2, {\rm an}})$ and the map $\bar{f}^{\rm MSBJ}$ between the boundaries (the dual intersection complexes) 
are $[0,1]^{2}$ projecting down to $[0,1]$. Note its fiber is also a segment, which coincides with the dual intersection complex of fibers. 
\end{Ex}

\begin{Ex}\label{easy.example}
If we consider $X:=V(xy-s)\subset \mathbb{A}^{4}_{x,y,s,t}$ with the natural projection $f$ down to $Y=\mathbb{A}^2_{s,t}$ and set 
$D_Y:=(st=0)$, $D_X$ as its pullback. Then the map $\bar{f}^{\rm MSBJ}$ between boundaries can be seen as a triangle mapping to the interval where two vertices map to a same 
edge point while the other point maps to the other edge point of the interval. The fiber is either an interval or a point. 
\end{Ex}

A natural question occurs from the above nicely behaved examples 
as follows. For simplicity, from here, we come back to 
ordinary varieties category while leaving discussions for 
the general analytic setup in the appendix. 

\begin{Ques}[Fibre compatibility]\label{Ques}
In the setting of Proposition~\ref{functoriality}, when 
$X, Y$ are both (analytifications of) varieties, 
we consider any 
morphism from 
the spectrum of DVR $R$, say $\overline{\varphi_Y}\colon \Spec(R)\to 
Y$ which maps the generic point $\iota=\Spec(\Frac(R))$ 
inside $U_Y$. We write the closed point of $\Spec(R)$ as $p$ which we 
suppose to map inside $\Supp(D_Y)$. We consider the natural extension 
of $\varphi|_{\iota}$ to 
\begin{align}
\overline{\varphi_Y}^{\hspace{1mm}\rm MSBJ}\colon \Spec(R)\to \overline{U_Y}^{\hspace{1mm} \rm 
MSBJ}.
\end{align} 
Then {\it under certain appropriate condition}, 
does it hold that 
\begin{align}
(\bar{f}^{\hspace{1mm}\rm MSBJ})^{-1}(\overline{\varphi_Y}^{\hspace{1mm} \rm 
MSBJ}(p))
\end{align} 
is canonically homeomorphic to 
the dual intersection complex of 
$\bar{f}^{-1}(\overline{\varphi_Y}(p))$? 
\end{Ques}

Note that Examples~\ref{Atiyah.flop} and \ref{easy.example} give affirmative examples. 
On the other hand, without assumption $f$ is flat and proper, it is easy to see various counterexamples (which we omit). 
Also, if $U_{X}\to U_{Y}$ is not smooth, then there is the following counterexample. 

\begin{Ex}
If $X$ is a blow up of $\mathbb{P}^{1}_{x}\times \mathbb{A}^{2}_{s,t}$ along $(t=x=0)\simeq \mathbb{A}^{1}$, 
$f\colon X\to Y:=\mathbb{A}^{2}_{s,t}$ is a natural projection, 
$D^{(1)}_{Y}:=(s=0)$, $D^{(1)}_{X}:=f^{*}D^{(1)}_{X}$. Then obviously 
we have trivial dual intersection complexes of $D_{X}^{(1)}$ 
and $D_{Y}^{(1)}$, hence at $p=(0,0)$ above speculation in 
Question~\ref{Ques} is violated. On the other hand, 
if we put $D_{Y}^{(2)}$ as $D_{Y}$ of Example~\ref{easy.example}, 
then the situation is completely that of Example~\ref{easy.example} 
hence the speculation holds. 
\end{Ex} 

Therefore, we are naturally led to temporarily 
assume $f|_{U_{X}}$ is smooth (or at least with normal fibers) and ask the same question \ref{Ques} again. 
Nevertheless, the assumption is still not enough! 

\begin{Ex}[Type II and III K3 degeneration as a counterexample]
Here is a much subtler counter-example to the speculation of 
Question~\ref{Ques} of 
degenerations of K3 surfaces including both 
Type II case and Type III case. 

We consider the closed subscheme 
$\mathcal{Y}'$ of $\mathbb{P}_{X,Y,Z,W}^{3}\times \mathbb{A}_{s,t}
^{2}$
defined by the vanishing of 
\begin{align*}
&(XZ+sQ_{1}(X,Y,Z,W)+tQ_{2}(X,Y,Z,W))\\  
&\times(ZW+sQ_{3}(X,Y,Z,W)+tQ_{4}(X,Y,Z,W))\\ 
&+stL_{1}(X,Y,Z,W)L_{2}(X,Y,Z,W),\\
\end{align*}
where $Q_{i}$ (resp., $L_{i}$) 
are general quadric (resp., linear) homogenous polynomials 
of variables $X,Y,Z,W$. Via projection, 
we have a natural morphism 
$\pi'\colon \mathcal{Y}\twoheadrightarrow \mathbb{A}_{s,t}^{2}$. 

Denote the open subset of $\mathcal{Y}'$ 
defined by $L_{i}\neq 0$ as $\mathcal{U}'$ and 
write the de-homogenization of $Q_{i}|_{\mathcal{U}'}$ 
as a polynomial $q_{i}$ 
of $\frac{X}{L_{i}}, \cdots, \frac{W}{L_{i}}$. 
Then they natural give a morphism from $\mathcal{U}'$ 
to $\mathcal{Y}$ of Example~\ref{Atiyah.flop} 
defined as $(q_{1},q_{2},s,t)$. Then we take the fiber product 
with respect to the small resolution of conifold singularity 
$\mathcal{X}_{i}\to \mathcal{Y}$ as 
Example~\ref{Atiyah.flop} and write 
$\mathcal{U}_{i}:=\mathcal{U}'\times_{\mathcal{Y}}\mathcal{X}_{i}$ 
for $i=1,2$. Then $\pi'$ gives a morphism $\pi_{i}\colon
\mathcal{U}_{i}\to \mathbb{A}_{s,t}^{2}$. 

We claim that $\pi_{2}\colon \mathcal{U}_{2}\to \mathbb{A}_{s,t}^{2}$ 
with the boundaries $V(st)\subset \mathbb{A}_{s,t}^{2}$ 
and its pullback $D_{\mathcal{U}_{2}}$ 
do {\it not} satisfy the speculation in Question~\ref{Ques}. 
Indeed, it is not hard to see that 
the dual intersection complex of 
$D_{\mathcal{U}_{2}}$ is a square homeomorphic to $[0,1]^{2}$, 
given that $\mathcal{U}_{2}\to \mathcal{Y}$ is small 
so that the exceptional set is of codimension at least $2$. 
The continuous map from it 
to that of $V(st)\subset \mathbb{A}_{s,t}^{2}$ 
is a square mapping down to an interval $[0,1]$ via the projection. 
On the other hand, 
over the closed point $s=t=0$ in $\mathbb{A}_{s,t}^{2}$, 
$\mathcal{Y}'$ has, hence so do $\mathcal{U}_{i}$s, 
the typical Type III degeneration - more explicitly a Zariski open 
dense 
subset (which contain all strata) of the 
$V(XYZW)\subset \mathbb{P}_{X,Y,Z,W}^{3}$. 
Therefore, this gives a counterexample to the 
speculation of Question~\ref{Ques}. 
\end{Ex}
Indeed, the author lately found that \cite[\S 8]{ACP} and 
\cite{CCUW} had essentially 
solved the problem for curve case (their $n=0$ case). 

\begin{Thm}[{\cite[Theorem 2]{CCUW} cf., also \cite[8.2.1]{ACP}}]
For the case of moduli of smooth projective curves of fixed genus $g\ge 2$ 
$\mathcal{M}_{g}$, the Question~\ref{Ques} has affirmative answer. 
\end{Thm}

\begin{proof}
It is well-known (probably since \cite{KnM}) 
that $\mathcal{M}_{g,1}\to \mathcal{M}_{g,0}=
\mathcal{M}_{g}$ defined by forgetting the marked point, at the 
Deligne-Mumford algebraic stack level, is nothing but the universal curve. 
Thus, the assertion readily follows from \cite[Theorem 2]{CCUW} 
restricted to charts of cone complex covering it (recall that 
\cite{CCUW} works over geometric stack of cone complexes), 
$\mathcal{M}_{g,n}^{\rm trop}$ in 
\cite{CCUW} is the incidence complex of the (stack theoritically 
simple 
normal crossing) Deligne-Mumford algebraic stack 
$\overline{\mathcal{M}_{g,n}}\setminus 
\mathcal{M}_{g,n}$ by \cite{ACP}. 
\end{proof}

The essential 
point is that each ``degenerating part'' (forming nodes) 
is independent at 
the discriminant locus $\overline{\mathcal{M}_{g,n}}\setminus 
\mathcal{M}_{g,n}$ (which is the reason why it is snc at smooth charts). 

Furthermore, we observe the abelian varieties analogue also holds. 

\begin{Prop}
For the 
Namikawa-Alexeev-Nakamura's compactification of moduli stack of principally polarized abelian varieties $\overline{\mathcal{A}_{g}}^{\rm AN}$ (\cite{Nam, AN, Ale02, 
Nak}) with its universal family of the polarized degenerated abelian varieties (\cite{AN, Ale02, Nak})\footnote{its name depends on the above literatures}, 
the Question \ref{Ques} has affirmative answer. 
\end{Prop}

\begin{proof}
This is essentially known (cf., \cite[\S 9 B]{Nam}). 
Indeed, {\it loc.cit} the definition 9.12 introduces a cone decomposition, called ``the mixed decomposition", of 
$$\{(y,z)\mid z\in V\otimes \R, y: \text{inner product on }V\otimes \R\},$$ 
where $V\simeq \Z^g,$ and shows that the fibers over $y$ is union of the corresponding Delaunay cells. 
Dividing by the discrete group of affine transformations 
$GL(V)\ltimes V$, we conclude the desired assertion. 
\end{proof}

Motivated by this, we guess that Question \ref{Ques} has affirmative answers for more general weak K-moduli spaces: 

\begin{Conj}\label{ex.family}
For a universal family over the various (other) weak K-moduli stacks 
\eqref{weak.Kmoduli.conj} 
such as  
Shah's compactification of moduli stack of polarized K3 surfaces of degree $2$ $\overline{\mathcal{F}_{2}}^{\rm Shah}$ (\cite{Shah}), 
and Alexeev-Engel-Thompson's compactification of the same moduli stack (\cite{AET}), the above Question~\ref{Ques} holds affirmatively. 
\end{Conj}
We further conjecture that the obtained family of tropical varities coincides with those parametrized in our previous works \cite{Od.Mg}, \cite{Od.Ag}, \cite{OO18} and \cite{CCUW}. 
\footnote{However, for $M_{g}$ case, remember that the Morgan-Shalen type topology does {\it not} coincides with the Gromov-Hausdorff topology given on $\overline{M_{g}}^{\rm T}$. cf., \cite{Od.Mg}}
We observe 
another affirmative direction to the above question \ref{Ques}. 
\begin{Prop}\label{toroidal.Q2}
In the setup of \eqref{functoriality} and Question \ref{Ques}, 
if $U_{X}\subset X$ and $U_{Y}\subset Y$ are toroidal pair of 
varieties, and the morphism $X\to Y$ is also toroidal with respect to the toroidal structures of $(X,X\setminus U_{X})$ and 
$(Y,Y\setminus U_{Y})$, then Question \ref{Ques} holds. 
\end{Prop}
Hence, in particular, if the 
examples in Conjecture \ref{ex.family} are toroidal families over the moduli stack, 
the desired assertions would follow. 
\begin{proof}
One can localize the problem to when 
$(X,X\setminus U_{X})$ and 
$(Y,Y\setminus U_{Y})$ are both toric  with 
torus equivariant morphism $X\to Y$ with respect to a 
morphism of the algebraic tori. Then, the fibers over closed points 
of $Y\setminus U_{Y}$ are stable toric varieties in the sense of 
\cite{Ale02}. This follows from \cite[2.1.11]{HLY} 
combined with \cite[\S 3]{Ale1}. Indeed, 
the irreducible components of the fibers are 
toric varieties for the relative stars \cite[2.1.9]{HLY} 
whose natural union is nothing but the fiber of 
the map between the fans. Hence, 
Question\ref{Ques} holds for this situation. 
\end{proof}

Hence the recent deep results on birational geometry of  
``toroidalization'' \cite{AK,ATW} may well be effective 
as in the proof of Theorem \ref{veryweak.Kmoduli} or \cite{KX2}.



\appendix
\section{Analytic Morgan-Shalen construction \\ 
and Dual intersection complexes}\label{appendixA}

This appendix reviews and extends 
both the dual intersection complex and 
the Morgan-Shalen type compactification  
(\cite{MS84, BJ}, \cite[Appendix]{TGC.II}) 
especially to analytic (including non-archimedean) 
general setting. 
This is logically used only in a few places of the main contents 
such as \S \ref{lim.toric.sec},  \eqref{functoriality}, among others, 
but we put it here partially as a preparation for future further use, 
and as a review of the original theory. 

The original Morgan-Shalen compactification \cite{MS84} compactifies a complex variety by attaching certain subset of cell complex, which \cite{MS84} 
applied to the character varieties and study of topology of manifolds. 
This idea is now being expanded in more modern contexts after inspiration coming from the attempt to understand the mirror symmetry geometrically 
(cf., e.g.  \cite{KS}, \cite{GS}, \cite{BJ}). Here, 
we review and prepare 
a further extended version of Morgan-Shalen type 
compactification, 
especially to include non-archimedean setting. 

In this texts, we first review the construction in complex algebraic setting, with some technical improvements, 
and then we later extend to general Berkovich analytic setting in the following subsection. 


\subsection{Review of (Algebraic) Dual intersection complex}

Started in the classical theory of dual graph of curves in surfaces, followed by the 
work of Kulikov, Pinkham-Persson \cite{Kul, PP} 
on the degeneration of surfaces, there has been a lot of 
works on dual intersection complexes (combinatorial data) for the normal crossing varieties (cf., \cite{Kul, PP, KS, GS, Step, dFKX, KX, BJ, Thuillier} and \cite[Appendix]{Od.Ag} etc). 
Here we discuss at more general setting. 

We keep the assumption that $k$ is algebraically closed. 

\begin{Ass}[Algebraic $\Q$-Cartier-ness$+$codimension assumption]\label{assumption.alg}
$X$ is a normal $n$-dimensional variety and $D:=X\setminus U$ is a finite union of irreducible {\it Cartier} divisors, 
i.e., all the components are (algebraically) $\Q$-Cartier, and the intersection of those $k$ irreducible components 
are of pure codimension $k$ for any $1\le k\le n$. 
\end{Ass} 

So far, people restricted the attention to normal crossing case or dlt($=$divisorially log terminal) case. 
Nevertheless, we can still {\it define} the {\it (algebraic) dual intersection complex} of $D$ denoted by $\Delta^{\rm alg}(D)$ as \cite[\S~2]{dFKX} do. 
{\it Op.cit} (Definition 8) assumes that the intersection of the irreducible components are irreducible but such requirements are not actually necessary for the definition 
once we associate a $l$-simplex to any {\it irreducible component} of $l$ irreducible components of $D$ and do the same inductive construction of $\Delta^{\rm alg}(D)$. On the other hand, 
the following easy lemma gives an efficient way to relate to 
the situation under the controll of log discrepancies. 
(This reminds the author of more difficult variant in 
a rather converse direction in \cite{FG}.)

\begin{Lem}[Log discrepancy controll]\label{cover.lemma}
Under the assumption~\ref{assumption.alg}, 
take an irreducible component $Z$ of the intersection of 
some irreducible components of $D$, with maximum 
possible ${\rm codim}(Z\subset X)$ which we denote as $l$. Write the 
generic point of $Z$ as $\eta_{Z}$ and 
the prime divisors containing $Z$ as $D_{1},\cdots,D_{l}$. 
Then we can take 
a Zariski open neighborhood $X'$ of $\eta_{Z}$ 
and equidimensional morphism $f\colon X'\to \mathbb{A}^{l}$ 
such that ${\rm Supp}(f^{*}({\rm div}(z_{i})))=D_{i}\cap X'$ 
for all $i$. 

In particular, if $Z$ is $0$-dimensional (i.e., a 
closed point) then the germ $Z\in X'$ is log canonical. 
In general, if we restrict $X$ to its Zariski open subset $X'$ 
without missing any strata, we obtain that 
$(X',D_{X}|_{X'})$ is log canonical. 
\end{Lem}

\begin{proof}
Take a set of regular functions $f_{1},\cdots,f_{l}$ on $X'$ 
such that $f_{i}$ defines the multiple of $\mathbb{Q}$-Cartier 
divisor $D_{i}$ around $\eta_{Z}$. Then consider the morphism 
$f:=(f_{1},\cdots,f_{l}):X'\to \mathbb{A}^{l}$ and 
take a Zariski open subset $Y'$ of $\mathbb{A}^{l}$ 
where $f$ is equidimensional and replace $X'$ by 
$f^{-1}(Y')$. Then we get the first assertion. 

For the second assertion, it follows from that 
$X'$ is log crepant to $Y'$ with appropriate 
effective boundary which encodes the ramification of $f$ in 
codimension $1$. Actually, from the proof, with an 
appropriate explicit effective divisor $D_{X'}$ 
(with standard coefficients) supported 
on $\cup D_{i}$, it follows that $(X',D_{X'})$ is log canonical. 
\end{proof}

Let us see some simple examples. 

\begin{Ex}
\begin{enumerate}
\item Let us set $X=(z^{d}=xy)\subset \mathbb{A}^{3}$, with boundaries 
$D_{1}=(x=0), D_{2}=(y=0)$, which maps via natural projection 
to $\mathbb{A}^{2}$ with coordinates $x$ and $y$. Then the 
induced morphism between the boundary of the 
corresponding Morgan-Shalen-Boucksom-Jonsson compactification 
is a homeomorphism between the segment to the segment. 
\item Let us set $X=\mathbb{A}^{2}$, with 
boundaries 
$D_{1}=(y-x^{m}=0), D_{2}=(y=0)$, mapping to 
$\mathbb{A}^{2}_{s,t}$ with the coordinates $s,t$ via 
$s:=y, t:=y-x^{m}$. Then although the morphism is not 
isomorphism, we similarly conclude that 
the 
induced morphism between the boundary of the 
corresponding Morgan-Shalen-Boucksom-Jonsson compactification 
is a homeomorphism between the segment to the segment. 
\end{enumerate}
\end{Ex}

Now, we re-discuss and generalize 
the functoriality a little more carefully following \cite[Appendix]{Od.Ag} and 
\S~\ref{fiber.mod.sec}. 


\begin{Prop}[Functoriality]\label{functoriality}
Let $k=\C$ and $f\colon X\to Y$ be a $k$-morphism between $k$-analytic spaces, $D_{X}$ be a divisor on $X$, 
$D_{Y}$ on $Y$, both satisfying the Cartier assumption \ref{assumption.an} 
(in the appendix). 
Consider the MSBJ compactification of $U_{X}:=X\setminus {\rm Supp}(D_{X})$ 
and $U_{Y}:=Y\setminus {\rm Supp}(D_{Y})$ accordingly. Then $f|_{U_{X}}$ continuously extend to a continuous map 
$$\overline{U_{X}}^{\rm MSBJ}(X)\to \overline{U_{Y}}^{\rm MSBJ}(Y),$$ 
which we denote by $\bar{f}^{\rm MSBJ}$.
\end{Prop}
\begin{proof}The proof is very similar to that of Proposition~\ref{functoriality.alg} (and \cite[A.15]{Od.Ag}). 
We take a net of affinoids domains of $X$ (resp., $Y$) denoted as $\{U_{i}\}_{i}$ (resp., $V_{i}$) where $U_{i}$ maps to $V_{i}$ for fixed $i$ and 
each $U_{i}$ or $V_{i}$ contains only one strata of biggest codimension in it. Then $\Delta(D_{X}\cap U_{i})$ and $\Delta(D_{Y}\cap V_{i})$ 
are both closed subset of simplices, from the definition. Note that for any $i,j$, $\Delta(D_{X}\cap U_{i}\cap U_{j})$ is a closed subcomplex of both 
$\Delta(D_{X}\cap U_{i})$ and $\Delta^{\rm alg}(D_{X}\cap U_{j})$ because the strata inside $\Delta^{\rm alg}(D_{X}\cap U_{i}\cap U_{j})$ is 
closed under specialization. The same for $V_{i}$s. 
As discussed in \cite[A.15]{Od.Ag}, $f$ induces a natural map from 
$\Delta(D_{X}\cap U_{i})$ to $\Delta(D_{Y}\cap V_{i})$, which glues continuously to the original continuous map $X\setminus D_{X}$ to $Y\setminus D_{Y}$. 
They glue at the closed 
subset $\Delta(D_{X}\cap U_{i}\cap U_{j})$s. Therefore, we obtain a continuous map from 
$\Delta(D_{X})$ to $\Delta(D_{Y})$ and whole $\overline{U_{X}}^{\rm MSBJ}(X)$ to $\overline{U_{Y}}^{\rm MSBJ}(Y)$ as well. 
\end{proof}

\begin{Rem}
At first sight, this Proposition~\ref{functoriality} 
could look like giving a broad extension of 
\cite[A.15]{Od.Ag} 
but note that, by the arguments of our lemma~\ref{cover.lemma}, 
at least semi-log-canonicities of dense open subsets of 
$(X,D_{X})$ or $(Y,D_{Y})$ are implicitly assumed. 
\end{Rem}

\begin{Rem}
For $U\subset X$, which 
are the analytifications of complex varieties, 
if we replace $X$ by $X'$ which dominates the original $X$ (while 
preserving $U$), then we have a continuous map 
$\overline{U^{\rm MSBJ}(X')}
\to \overline{U^{\rm MSBJ}(X)}$ preserving $U$. 
The projective limit of these 
Morgan-Shalen type compactification $\overline{U^{\rm MSBJ}(X)}$s,  
where $X$ runs over all $X$, satisfying the 
assumption \ref{assumption.an} should coincide with 
the compactification of \cite{PM}. 
At least when $U$ is affine, there is a construction by 
Favre (unpublished) which the author fortunately 
had a chance to study\footnote{I appreciate S.Boucksom, C.Favre and 
M.Jonsson for this communication}, 
and the coincidence in this case 
can be proven similarly as \cite[4.12]{BJ}. 
In \cite{PM} the authors even put locally ringed space structure on 
the compactification, discusses coherent sheaves for it and 
proved a GAGA type theorem. 
\end{Rem}

\begin{Rem}
A related result in special situation is \cite[5.2.1, 6.1.6.]{ACP}. 
\end{Rem}

\begin{Rem}Also there is an analogous observation 
in ``classical'' (embedded) tropical geometry setting of this sort 
cf., e.g, \cite[\S6]{RSS}. 
\end{Rem}


\subsection{Projectivizing affine structure}\label{proj.str.sec}

In usual tropical algebraic geometry, affine structure plays a central role either explicitly or implicitly. Here, 
we introduce a {\it projectivized} version, i.e., 
an analogue of affine structure 
obtained by dividing by the action of $\mathbb{R}_{>0}$. 
These are essentially not so new and has been implicit in literatures in the sense that various examples have appeared, 
but we make a systematic introduction. 
For instance, 
this can be seen as a tropicalized version of 
{\it normalized Berkovich space} in the sense of Fantini \cite{Fan}. 

\begin{Def}
Suppose $k$ is a non-archimedean field and 
let $0\in V\simeq \mathcal{M}(\mathcal{A})$ be a germ of $k$-affinoid with the isomorphism as Banach $k$-algebra 
$\mathcal{A}\simeq k\{\frac{T_1}{r_1},\cdots,\frac{T_N}{r_N}\}/I_V$ with $r_i\in\R_{>0}$ and ideal $I_{V}$. Equivalently, 
$V$ can be regarded as a subspace in $\mathbb{A}_k^{N,{\rm an}}$. 
We would call this germ $0\in V$ with the additional data, 
{\it framed or embedded germ} but 
we may simply write $0\in V$ if the rest is obvious from the context. 
 We set the coordinates of $\A^{N}$ as $z_i$s. 
 
 Now we consider a natural subset of the real projective space $\mathbb{P}_{\R_{\ge 0}}^{N-1}$ as $(\R_{\ge 0}^{N}\setminus \{\vec{0}\})/\R_{>0}$ and 
call it the {\it projective simplex} of dimension $N-1$. 
Note that the inclusion 
$$\Delta_{\vec{1}}:=\{(x_{1},\cdots,x_{N})\in \R_{\ge 0}^{N}\mid \sum_{i} x_{i}=1\}
\hookrightarrow (\R_{\ge 0}^{N}\setminus \vec{0}),$$
composed with the projection to $\mathbb{P}_{\R_{\ge 0}}^{N-1}$ 
is homeomorphism and denote it as 
$$\varphi_{(1,\cdots,1)}\colon \Delta_{\vec{1}}
\xrightarrow{\simeq}
\mathbb{P}_{\R_{\ge 0}}^{N-1}.$$ However, we do {\it not} really respect this 
``artificial'' map nor the 
affine structure on the projective simplex induced by this, 
in general. 

Then we define 
the {\it projective tropicalization map} of framed germ $0\in V$ as 
\begin{align*}
{\rm PTrop}_{V}\colon    V           &\cap \{|z_{i}|<1 \forall i\}&\to    &   \hspace{1cm}     \mathbb{P}_{\R_{\ge 0}}^{N-1} \\ 
                                                       & 
                                                       \hspace{2mm}{\rotatebox{90}{$\in $}}    &               & 					    \hspace{11mm}{\rotatebox{90}{$\in $}}  \\
                                                       &\hspace{2mm} x                                              &\mapsto&             \hspace{3mm}             [\cdots:  -\log|z_{i}|_x:\cdots].     \\ 
\end{align*}
\end{Def}

\begin{Def}
Then we set the {\it projective tropicalization set} ${\rm PTrop}(0\in V)$ as the limit set 
\begin{align*}
\{\lim_{i} {\rm PTrop}_{V}(x_{i})\in\mathbb{P}_{\R_{\ge 0}}^{N-1}\mid x_{i}\in V (i=1,2,\cdots), \lim_{i}x_{i}=\vec{0}\}.
\end{align*}
\end{Def}

Then we have the following characterizations, 
which can be seen as a projective analogue of the 
``fundamental theorem'' by Kapranov in tropical geometry with 
embedded formalism (cf., e.g., \cite[3.2.5]{MaS}). 

\begin{Thm}[Projective tropicalization]\label{PTrop.char}
The following subsets of $\mathbb{P}^{N-1}_{\R_{\ge 0}}$ are the same: 
\begin{enumerate}
\item \label{PT1}
${\rm PTrop}(0\in V)$, defined above  

\item \label{PT2}
the quotient of 
$\bigcap_{0\neq f\in I_{V}} V({\rm trop}(f))$ 
by the action of $\mathbb{R}_{>0}$, where ${\rm trop}(-)$ means the usual (non-archimedean) tropicalization of the $k$-polynomial, 
${\rm trop}(-)^{(1)}$ is the degree one term of 
${\rm trop}(f)$ (i.e., constants discarded), and 
$V(-)$ denotes the corresponding tropical hypersurface 
in $\mathbb{R}^{N}$. 

\item \label{PT3}
the set of positive directions in 
${\rm Trop}(V)$ i.e., 
$$\{\vec{v}(\neq \vec{0})\in \R_{\ge 0} \mid \R_{\ge 0}\vec{v}+
\vec{w}\subset {\rm Trop}(V) \text{ for } \exists \vec{w}\}.$$ 
\end{enumerate}

In particular, ${\rm PTrop}(0\in V)$ is 
a projectivization of a piecewise linear set. 
\end{Thm}

\begin{proof}
The equivalence of \eqref{PT2} and \eqref{PT3} 
is a standard exercise. \eqref{PT2}$\supset$\eqref{PT3} is 
immediate from the definitions and \eqref{PT2}$\subset$\eqref{PT3} 
is more nontrivial but it still holds since for 
 each unbounded polyhedron $P$ in \eqref{PT3}, 
having $\R_{\ge 0}v_{i}$s as the edge half-lines, 
any (tropical) term 
of ${\rm trop}(f)$ whose degree $1$ homogeneous part 
is not minimized at $P$, is 
bigger than 
${\rm trop}(f)$ at the region 
$P+l \sum_{i} v_{i}$ for $l\gg 0$. 
The equivalence of \eqref{PT1} and \eqref{PT2} (or \eqref{PT3}) 
follows from the ``Fundamental theorem" \cite[Theorem 3.2.5]{MaS} 
(originally due to Kapranov). 
\end{proof}

Roughly speaking, the above says that ${\rm PTrop}$ is the ``tangent cone of ${\rm Trop}$ at  infinity in positive(first quadrant) direction" as an 
analogue of the 
tangent cone at infinity in Riemannian geometry. 
From the above  \eqref{PTrop.char}, 
it easily follows that: 
\begin{Prop}
If $I_{V}=(f)$, denote the normal fan of the Newton polytope of 
$f$ as ${\rm Newt}(f)$ and the set of its rays inside 
the boundary of the cones as $\partial{\rm Newt}(f)$. 
Then $\PTrop(0\in V)$ is 
$$(\partial{\rm Newt}(f)\cap \R_{>0}^{N})/\R_{>0}.$$
\end{Prop}

\begin{Ex}
If $N=2$, and $I_{V}=(f)$ with degree $d$ polynomial 
$f$ of $z_{i}$s, then ${\rm PTrop}(0\in V)$ is a finite set 
with order at most $g(d)$ where $g(1)=g(2)=1, g(3)=g(4)=2, 
g(5)=g(6)=g(7)=3$ for instance. 
\end{Ex}

\begin{Ex}
For any {\it homogeneous} $k$-polynomial $f$ of degree $d$, 
it is easy to see that 
$\varphi_{\vec{1}}^{-1}(\PTrop(0\in V=V(f)))\subset \Delta_{\vec{1}}$ 
is the set of rays of a fan with the center $\overrightarrow{(1,\cdots,1)}$ 
in any way. 
\end{Ex}

\subsection{General analytic (Berkovich) setting}

Now we move on to analytic extension. Let us consider the following examples. 

\begin{Ex}\label{nodal.rational.curve}
Let $X$ be $\mathbb{P}^{2}$ and $D$ is the nodal rational cubic curve. Then $U:=X\setminus D$ is an easy typical example of so-called 
{\it log Calabi-Yau surface}. In this case, the 
{\it algebraic} dual intersection complex whose definition was briefly recalled in the previous subsection, 
is just a single closed point. However, it is sometimes more natural to consider a {\it loop} i.e., topologically $S^{1}$ with one vertex and one edge, 
as the natural candidate for ``more correct'' dual intersection complex. 
\end{Ex}

To remedy the above problem, related to monodromy, we introduce the following analytic extension. We still keep the assumption 
that $k$ is a non-archimedean field. 

Recall that the sheaf of Cartier divisor over any locally 
ringed space 
is defined (similarly to the case of schemes), simply 
as the sheaf of meromorphic functions divided by that of invertible holomorphic 
functions. 
In particular, the case of complex manifolds, the 
theory of divisors is of course well-established and frequently used. 
Here, in the setting of Berkovich analytic spaces, we use the same theory of Cartier divisors with respect to the 
(weak) topology, not G-topology. 

How about Weil-divisors? For general complex spaces, one can define in a similar manner (cf., e.g., \cite[6.7]{Dem}). 
In the non-archimedean literatures, the case when $X$ is smooth and one dimension (often just an analytification of smooth curve) 
is treated (cf., \cite{Ber93} and later) and we can similarly think of formal linear combination of codimension $1$ closed irreducible analytic spaces 
following the notion of (irreducible) closed analytic subspace in \cite[\S~3.1]{Ber90}. 
In this paper, we say a closed analytic subspace $Z$ of $X$ is 
Cartier divisor, 
if the corresponding coherent ideal sheaf $I_Z$ is invertible, i.e., for any $x\in Z$, there is  
a small enough open neighborhood of $x$, say $U\subset X$, such that 
$I_Z|_U$ is generated by a regular element of $\Gamma(U,\mathcal{O}_X)$. 

\begin{Ass}[Analytic  $\Q$-Cartier-ness$+$codimension assumption]\label{assumption.an}
$X$ is a normal (irreducible) $n$-dimensional $k$-strict 
analytic space and $D:=X\setminus U$ is a finite union of (supports of) irreducible {\it Cartier} divisors $D_i$s i.e., 
there is a covering by strict affinoid domains $\{V \simeq \M(A_{V})\}_{V}$and $z_{1,V},\cdots,z_{k_V,V}\in A_{V}$ 
such that $$D_i\cap V=\{y\in V\mid |z_{i,V}(y)|=0\}.$$ 
Furthermore, we assume that the intersection of any $k$ among $D_i$s 
has pure codimension $k$ for all $1\le k\le n$. 
\end{Ass} 

In this setting, we will define a {\it (generalized) Morgan-Shalen-Boucksom-Jonsson partial compactification} of 
$U$ as a topological space $$(U\subset) \overline{U^{\rm MSBJ}}(X),$$ 
as well as its boundary - the {\it dual intersection complex} $\Delta^{\rm anal}(D)=\Delta(D)$. 
Our construction generalizes previous \cite{MS84} (affine case), \cite{BJ} (smooth with normal crossing boundary case), \cite[Appendix]{Od.Ag} (divisorially log terminal case). 
As in \cite[Appendix]{Od.Ag}, we can also easily generalize it to Deligne-Mumford stacks whose \'etale charts (covering) and the pullback of those 
boundaries satisfy the above condition \ref{assumption.an}. 
Leaving such verbatim stacky extension to the readers, we give the definition when $X$ is an analytic space (under assumption \ref{assumption.an}) similarly to 
\cite{MS84, BJ} and \cite[Appendix]{Od.Ag}. 
The construction is as follows: 
\begin{Const}\label{anal.MSBJ}
\begin{Step}[Local ambient space of boundary]
For strict $k$-affinoid domain $V\simeq \M(A_{V})$ of $X$ such that 
$\emptyset \neq (V\setminus U)=\cup_{1\le i\le k_{V}}(z_{i,V}=0)$ 
where $z_{i,V}$ are elements of $A_{V}$ and 
each $z_{i,V}=0$ gives distinct irreducible divisor of $V$ 
with multiplicity one. 
This is possible because of the above assumption \ref{assumption.an}. 
Then we consider 
$$\widetilde{\Delta_{V}}(U\subset X):=\mathbb{P}_{\R_{\ge 0}}^{k_V}=(\R_{\ge 0}^{k_{V}}\setminus \{\vec{0}\})/\R_{>0},$$

(or simply $\widetilde{\Delta_{V}}$ when $U, X$ are obvious from the context) 
where each $i$-th real (projectivized) coordinate will be connected to the local function $z_{i,V}$. 
An important remark is that $\widetilde{\Delta_{V}}$ does {\it not} have natural affine structures, 
but only able to regard as a subset of real projective space, i.e., 
the projective structure  in the sense of \S~\ref{proj.str.sec}. 
\end{Step}

\begin{Step}[Tropicalization and Local boundary]\label{trop.map}
Under the above setting, 
we define the tropicalization map as usual as 
\begin{align*}
{\rm Trop}_{V}\colon & (V\cap U)&\cap \{|z_{i,V}|<1 \forall i\}&\to         & \widetilde{\Delta_{V}}&(U\subset X) \\ 
                                       &                & {\rotatebox{90}{$\in $}}    &               & 					    &{\rotatebox{90}{$\in $}}  \\
                                       &                & x                                              &\mapsto&                          [\cdots: & -\log|z_{i,V}|_x:\cdots].     \\ 
\end{align*}
and $\Delta_{V}(U\subset X)\subset \widetilde{\Delta_{V}}(U\subset X)$ as 
$$\{\lim {\rm Trop}_{V}(x_{i})\in \widetilde{\Delta_{V}}(U\subset X)\mid x_{i}\in (V\cap U) (i=1,2,\cdots), \lim_{i}x_{i}\in V \cap D\}.$$ 
From diagonal arguments, $\Delta_{V}(U\subset X)\subset \widetilde{\Delta_{V}}(U\subset X)$ is a closed subset since $V$ is compact. 
It is easy to see that $\Delta_{V}(U\subset X)$ is the projective tropicalization set 
${\rm PTrop}(V\cap U)$ we introduced in \S\ref{proj.str.sec}. 
Note that this construction does not depend on $z_{i,V}$s 
since its replacement only changes $z_{i,V}$ by 
unit, hence the change of $-\log|z_{i,V}|_{x}$ 
is bounded above by a constant depending on $V$.  
\end{Step}

\begin{Step}[Global ambient space of boundary]

For affinoid subdomains $V_{i}\subset X (i=1,2)$ both satisfying above requirements and $V_{1}\cap V_{2}\neq \emptyset$, we can naturally consider the projectivization of a natural real linear map: 
$$\widetilde{\Delta_{V_{1}\cap V_{2}}}(U\subset X)\to \widetilde{\Delta_{V_i}}(U\subset X)$$ 
which maps a vertex $\widetilde{\Delta_{V_{1}\cap V_{2}}}(U\subset X)$ corresponding to a prime analytic divisor $D_{j}\cap V_{1}\cap V_{2}$ to 
the vertex of $\widetilde{\Delta_{V_{i}}}(U\subset X)$ corresponding to the same prime divisor $D_{j}\cap V_{i}$. We call this map gluing map and denote it as $\varphi=\varphi_{V_{1}\cap V_{2}, V_{i}}$. 
It is easy to see $\varphi$ maps $\Delta_{V_{1}\cap V_{2}}(U\subset X)$ into $\Delta_{V_{i}}(U\subset X)$ from the definition. 
Then we consider the equivalence relation $\sim$ on 
$\sqcup_{V} \Delta_{V}(U\subset X)$ generated by the identification of the source point and the target point of gluing maps $\varphi$s. 
Then we set: 
$$\Delta(U\subset X):=\Bigl(\bigsqcup_{V\subset X\text{: affinoid domains}} \Delta_{V}(U\subset X)\Bigr)/\sim.$$
\end{Step}

\begin{Step}[Topology at local level]
For each $V\subset X$ satisfying above requirement, we put topology on 
$$\overline{(V\cap U)}^{\rm MSBJ}(V):=(V\cap U)\sqcup (\Delta_{V}(U\subset X)),$$
as the same way as \cite[p415]{MS84} or \cite[\S~2.2, Def2.3]{BJ} using ${\rm Trop}_{V}$. 
It is easy to see the construction does not depend on the choices of local defining equations $z_{i,V}$s. 
\end{Step}

\begin{Step}[Topology at global level]

We define the {\it Morgan-Shalen-Boucksom-Jonsson partial compactification} of $X$, which we write 
$\overline{U}^{\rm MSBJ}(X)$, as the colimit topological space of 
$\overline{(V\cap U)}^{\rm MSBJ}(V):=(V\cap U)\sqcup (\Delta_{V}(U\subset X))$ 
for all affinoid subdomains $V\subset X$s. From the construction, it has a tautological open subset which can be identified with $U$. 
Then, the {\it analytic dual intersection complex} is 
$$\Delta(D):=\partial \overline{U}^{\rm MSBJ}(X):=\overline{U}^{\rm MSBJ}(X)\setminus U.$$
\end{Step}

\end{Const}

As in \cite[Appendix]{Od.Ag}, the above construction also naturally extends to Deligne-Mumford stack pair $(\mathcal{X},\mathcal{D})$ which 
satisfies the same assumption~\ref{assumption.an}. We omit the details as it is verbatim. 

A difference with archimedean situation is that 
the topological dimension of $\Delta(D)$ is 
at most $\dim(U)-1$ (which indeed is attained if $D$ is 
``maximally degenerate'' e.g. normal crossing divisor 
with $0$-dimensional strata). 

\begin{Prop}[$\Delta^{\rm alg}$ vs $\Delta$]\label{Delta.Delta}
For a given $U\subset X$, we have a natural continuous surjection $\Delta(X\setminus U)\twoheadrightarrow \Delta^{\rm alg}(X\setminus U)$. 
\end{Prop}
\begin{proof}
This easily follows the definitions, 
since given any Zariski open covering of $X$, 
we can  refine the covering of $X^{\rm an}$ by 
analytifying all open subsets, to 
(fine enough, if necessary) strict affinoid subdomains. 
We leave writing the details to readers (or myself in future). 
\end{proof}

For instance, in Example~\ref{nodal.rational.curve}, a loop shrinks to a point by the above map in Proposition~\ref{Delta.Delta}. 
In general, if $X$ is compact, e.g., analytification of proper scheme \cite[3.4.8 (cf., also 3.5.3)]{Ber90}, then $\bar{U}^{\rm MSBJ}(X)$ becomes compact by the above construction. 

\begin{Rem}[Non-archimedean symmetric space]

Recall that \cite[\S2]{OO} proved that in the setting over 
complex numbers, 
Satake compactification of adjoint type for 
locally Hermitian symmetric space coincides with 
Morgan-Shalen type compactification. It is natural to seek 
its non-archimedean analogue. 
However, the right formulation seems  more nontrivial to the author 
as we observe the following. 

Recall Berkovich gave a compactification of Bruhat-Tits building 
$\mathcal{B}(G,k)$ for semisimple algebraic group $G$ over 
non-archimedean field 
i.e., closure 
inside the analytification of flag variety 
along \cite[Theorem 5.5.1]{Ber90}, 
and its extension by \cite{RTW1, RTW2} which 
they call an analogue of Satake compactification. 
Restricting to an apartment of $\mathcal{B}(G,k)$, 
we observe that its closure inside the compactification of 
\cite{Ber90}, \cite{RTW1, RTW2} is of Kajiwara-Payne type 
compactification hence not compatible with 
any Morgan-Shalen-Boucksom-Jonsson compactification of 
$G^{\rm an}$. 
\end{Rem}


\section{Towards Satake-Baily-Borel type compactification}\label{SBB.sec}

Also we briefly discuss the possible analogue of
the Satake-Baily-Borel type compactification (\cite{Sat0, Sat2, BB}) for moduli of 
Calabi-Yau varieties which do {\it not} necessarily have a structure of locally 
Hermitian symmetric domain. 
The discussions in this section do not contain substantial results 
but rather only proposing conjectures and some observations, 
which is the reason of presentation as short notes in the appendix. 

Recall that 
the algebro-geometric meanings of the Satake(-Baily-Borel) compactification \cite{Sat0} 
of the moduli $A_{g}$ of principally polarized abelian varieties are well-understood 
(cf. e.g., \cite{Cattani}) 
which can be more geometrized  
by Grothendieck semiabelian reduction theorem (cf., e.g., \cite[4.4.1]{Chai}). 
Similar results are also known for K3 surfaces in somewhat weak sense 
cf., e.g., \cite{FriS}, 
i.e., the boundary parametrizes the graded sum of the weight filtration of the 
limit mixed Hodge structures and the only nontrivial boundaries, the modular curves, 
parametrize elliptic curves which are the minimal log canonical centers of the 
Kulikov type II degenerations. 

We also discuss possible differential geometric meaning, 
for $k=\C$ case,  
via Ricci-flat K\"ahler metrics, as parametrizing the limits while fixed {\it volumes}, at 
Question \ref{SBB.GH} (in the short appendix) which we hope to explore more in future. 

Recently, there is also an attempt of generalizing the 
Satake-Baily-Borel compactification \cite{GGLR} 
for the 
{\it image of periods map from the moduli}. 
Hoping to complement geometric perspective somewhat to their works, 
we would like to discuss another generalization in the following manner. 
The connection is unclear yet, mainly due to lack of 
general Torelli-type theorem for general K-trivial varieties 
which we hope to clarify in future. 

\begin{Conj}\label{SBB.conj}
For any moduli $M$ of polarized log-terminal Calabi-Yau varieties, 
under the assumptions \ref{Ass1}, \ref{Ass.log.gen.type}, 
we also have another compactification 
$M\subset \overline{M}^{SBB}$ as follows. 

\begin{enumerate}
\item \label{coincide.BB} (Baily-Borel compactification)
$\overline{M}^{SBB}$ is the log canonical 
model $M_{\rm lc}$ of $M$. More precisely, 
$M\subset \overline{M}^{SBB}$ is isomorphic to the quotient by $\Gamma$ of 
the log caninical model $M'_{\rm lc}$ of $(\overline{M'},D')$ 
in the setting of \S \ref{Kmoduli.sec}. Note that it does not 
depend on $(\overline{M'},D')$. 

In particular, if $M$ is uniformized by a Hermitian symmetric domain 
(e.g., when $X$ is abelian varieties or holomorphic symplectic manifold), 
then $\overline{M}^{SBB}$ is nothing but the Satake-Baily-Borel compactification of $M$ by \cite{Mum77}. 

\item \label{SBB.map}
For any weak K-moduli compactification 
$\overline{\mathcal{M}}$, there is a natural morphism 
$\overline{\mathcal{M}}\to \overline{M}^{SBB}$. 

\item \label{CM.ample} (Hodge-CM line bundle) 
The CM line bundle on $\overline{\mathcal{M}}$ 
(which is automatically a positive 
rational multiple of the Hodge line bundle) 
is a pullback of an ample $\Q$-line bundle on $\overline{M}^{SBB}$, 
by the morphism \eqref{SBB.map} above.  

 \item \label{fin.GGLR} 
In the case of moduli of strict
\footnote{the assumption ensures the coincidence of the 
augmented Hodge line bundle and the Hodge line bundle in \cite{GGLR}}  
Calabi-Yau varieties i.e., 
general fibers $X$ satisfies $H^{i}(\mathcal{O}_{X})=0$ for all 
$0<i<{\rm dim}(X)$, $\overline{M}^{SBB}$ has a natural finite 
morphism to the conjectural compactification of \cite[Conjecture 1.2]
{GGLR}, 
as an extension of the period map from $M$, 
and that the extended Hodge line bundle pulls back to 
the CM line bundle modulo taking positive tensor powers. 
\end{enumerate}

\end{Conj}

\begin{Rem}
In the case of 
polarized abelian varieties, K3 surfaces and (compact) 
hyperK\"ahler manifolds and their irreducible 
symplectic degenerations, 
the above conjectures are known to experts as fully confirmed 
(cf., e.g., \cite{Fjn.SBB}, \cite{FC90}, 
\cite{FriS}, \cite{Schu}, \cite[\S 8]{OO}, \cite{GGLR}). 
\end{Rem}

\begin{Rem}
Note that if the \cite[Conjecture 1.2]{GGLR} is true, 
\cite[Theorem 1.3.10]{GGLR} implies that above 
\eqref{CM.ample} for strict Calabi-Yau case would follow 
from \eqref{fin.GGLR}. 
\end{Rem}

To go into the depth of the above conjecture \ref{SBB.conj}, 
we assign the concept of minimal log canonical center, the role of a key player. 
They are ``canonical'' at least in some weak sense, for 
degenerating Calabi-Yau family, 
as we partially show. 
More precisely, at the moment we have the following, 
which refines the $\mathbb{P}^{1}$-linking theorem 
of Koll\'ar \cite{Kol11} to a generalized setup 
without specifying models, unlike 
{\it loc.cit}, 
by using previous lemmas. 

\begin{Prop}[Birational uniquenss of minimal log canonical centers]\label{lc.center.bir}
Fix a polarized klt Calabi-Yau family $(\mathcal{X}^{*},\mathcal{L}^{*})\to \Delta^{*}=\Delta\setminus 
\{0\}$ (as \S \ref{fiber.sec}) and its base changes for the $N$-th ramifying morphism 
$\Delta\to \Delta$ as $(\mathcal{X}^{[N],*},\mathcal{L}^{[N],*})$.

Minimal log canonical centers of the central fibers of dlt minimal models 
$(\mathcal{X}^{[N]},\mathcal{L}^{[N]})\to \Delta$ 
are all {\it birational} (we allow to chang $N$ and the models). \end{Prop}

\begin{proof}
We divide the arguments into three steps. 

\begin{Stp} ($\Q$-factorization) 
We consider dlt minimal model $\X'$ and its 
crepant $\Q$-factorialization $\X$. Since $\X$ is 
terminal by Lemma\ref{dlt.lem} \eqref{dlt.term} so that it is 
Cohen-Macaulay and in any case 
$\X_{0}$ satisfies Serre $S_{2}$ condition. 
Therefore, the strata of $\X$ cannot be an 
exceptional locus (contratable) to 
$\X'_{0}$. 

\end{Stp}
\begin{Stp} (Flops) 
For any two $\Q$-factorial dlt minimal models 
$\X_{1}, \X_{2}$, they are conneted by flops due to \cite{Kawamataflop}. Then we apply the same arguments as 
above Step 1 to the 
flopping contraction to check the assertion. 

For fixed $\X$, the assertion follows from the 
$\mathbb{P}^{1}$-linking theorem \cite{Kol11}. 

\end{Stp}
\begin{Stp} (Base change effect)
We show that for a dlt minimal model $\X\to \Delta$ and any 
positive integer $N$, there is 
a dlt minimal model $\X^{[N]}$ admissibly 
dominating the base change of $\X$ which satisfies the assertion. 
This follows directly from the actual construction of 
\cite[\S4]{ops} 
as partially reviewed in Proposition \ref{subdiv}. 
\end{Stp}
\end{proof}
Accordingly, we expect parametrization 
of minimal lc centers on the boundary of the 
Satake-Baily-Borel type compactification of Conjecture \ref{SBB.conj}. 
We introduce the following terminology: 
\begin{Def}\label{CY.motif}
We define a set 
$$\mathcal{M}_{CY(d)}:=
\{d\text{-dimensional klt log Calabi-Yau pairs}\}/\sim,
$$
where in the right hand side, $\sim$ means the equivalence 
relation generated by log 
crepant birational maps (``B-birational map'' in \cite{Fjn.Mthesis}). 
\end{Def}

\begin{Conj}[Parametrising minimal log canonical centers]
\label{SBB.conj2}
For any of the conjectural compactification 
$M\subset \overline{M}^{\rm SBB}$ of \eqref{SBB.conj}, 
there is a natural map 
$$\psi_{mlcc}\colon  
\partial 
\overline{M}^{\rm SBB}\to 
\bigsqcup_{0\le d<n}\mathcal{M}_{CY(d)}$$ 
such that the following holds: 

Take a polarized dlt minimal model 
$(\X,\mathcal{L})\to C\ni 0$, 
whose fibers are klt and parametrized by $M$ 
away from $0\in C$ 
so that corresponding to 
$\varphi^{o} \colon C\setminus \{0\}\to M$, 
consider the extension 
$\varphi\colon C\to \overline{M}^{\rm SBB}$. 
Then $\psi_{mlcc}(\varphi(0))$ is represented (in the sense of 
Definition \ref{CY.motif}) by 
the {\bf minimal log canonical center} of 
$(\X,\X_{0})$, with the natural different in the sense of Shokurov 
(it makes sense by Theorem \ref{lc.center.bir}). 
\end{Conj}

Also recall a fact on the 
abelian varieties case from \cite{TGC.II}:
\begin{Thm}[SBB compactification as GH compactification 
{\cite[\S 2.5.3, Corollary 2.14]{TGC.II}}]
Let us consider the moduli $A_{g}$ of $g$-dimensional 
principally polarized 
abelian varieties $(X,L)$ over $k=\C$, and 
associate the flat K\"ahler metrics on $X$ with the K\"ahler classes 
$2\pi c_{1}(L)$ so that the volume is always $1$. 
Then, the pointed Gromov-Hausdorff limits of those are 
(trivial) $\R^{r}$-fibrations over the minimal log canonical centers 
which are $g-r$ dimensional principally polarized abelian varieties, 
with the natural flat metrics, where $r$ denotes the torus rank 
of the degenerations. 

In particular, 
those pointed Gromov-Hausdorff limits are parametrized by the Satake-(Baily-Borel) compactification 
$\overline{A_{g}}^{\rm SBB}$ (\cite{Sat0}). 
\end{Thm}

Naturally this leads to the following question. 

\begin{Ques}[Satake-Baily-Borel compactification as 
Gromov-Hausdorff compactification?]\label{SBB.GH}
For the above situation, when $k=\C$, does 
$\partial \overline{M}^{\rm SBB}$ parametrizes certain informations 
of 
the (pointed) Gromov-Hausdorff limits of 
Ricci-flat K\"ahler spaces $(X,\omega_{X})$ 
where $(X,L)\in M$, $[\omega_{X}]=c_{1}(L)$, and 
${\rm diam}(-)$ denotes the diameter. 
\end{Ques}


\footnotesize \noindent
Contact: {\tt yodaka@math.kyoto-u.ac.jp} \\
Department of Mathematics, Kyoto University, Kyoto 606-8285. JAPAN \\

\end{document}